\documentclass[
               letterpaper,%
               twoside,%
               10pt%
               ]{article}

\usepackage[hmargin={1.0in,1.0in},
            vmargin={1.0in,1.0in},%
            headheight=14pt,
            ]{geometry}

\usepackage{amsmath}
\numberwithin{equation}{section}                        
\usepackage{amsthm}
\usepackage{amssymb}                                    
\usepackage{bbm}                                        
\makeatletter
\let\@upn\@iden
\makeatother
\newtheoremstyle{ToddTheorem}                           
  {3pt}
  {3pt}
  {\slshape}
  {}
  {\bfseries\upshape}
  {.}
  {.5em}
  {}

\newtheoremstyle{ToddDefinition}                        
  {3pt}
  {3pt}
  {\upshape}
  {}
  {\bfseries\slshape}
  {.}
  {.5em}
  {}

\makeatletter
\newcommand*{\c@theoremassumption}{\c@theorem}
\newcommand*{\p@theoremassumption}{\p@theorem}

\makeatother

\makeatletter
\newcommand*{\c@theoremconjecture}{\c@theorem}
\newcommand*{\p@theoremconjecture}{\p@theorem}

\makeatother

\makeatletter
\newcommand*{\c@theoremcorollary}{\c@theorem}
\newcommand*{\p@theoremcorollary}{\p@theorem}

\makeatother

\makeatletter
\newcommand*{\c@theoremdefinition}{\c@theorem}
\newcommand*{\p@theoremdefinition}{\p@theorem}

\makeatother

\makeatletter
\newcommand*{\c@theoremexample}{\c@theorem}
\newcommand*{\p@theoremexample}{\p@theorem}

\makeatother

\makeatletter
\newcommand*{\c@theoremfigure}{\c@theorem}
\newcommand*{\p@theoremfigure}{\p@theorem}

\makeatother

\makeatletter
\newcommand*{\c@theoremhypothesis}{\c@theorem}
\newcommand*{\p@theoremhypothesis}{\p@theorem}

\makeatother

\makeatletter
\newcommand*{\c@theoremlemma}{\c@theorem}
\newcommand*{\p@theoremlemma}{\p@theorem}

\makeatother

\makeatletter
\newcommand*{\c@theoremproposition}{\c@theorem}
\newcommand*{\p@theoremproposition}{\p@theorem}

\makeatother

\makeatletter
\newcommand*{\c@theoremremark}{\c@theorem}
\newcommand*{\p@theoremremark}{\p@theorem}

\makeatother

\makeatletter
\newcommand*{\c@theoremnotation}{\c@theorem}
\newcommand*{\p@theoremnotation}{\p@theorem}

\makeatother

\newtheoremstyle{SpringerTheorem}                           
  {3pt}
  {3pt}
  {\itshape}
  {}
  {\bfseries}
  {.}
  {.5em}
  {}
\newtheoremstyle{SpringerDefinition}                        
  {3pt}
  {3pt}
  {\rmfamily}
  {}
  {\bfseries}
  {.}
  {.5em}
  {}
\newtheoremstyle{SpringerExample}                        
  {3pt}
  {3pt}
  {\rmfamily}
  {}
  {\itshape}
  {.}
  {.5em}
  {}

\theoremstyle{SpringerTheorem} %
\newtheorem{theorem}{Theorem}[section]

\newtheorem{lemma}[theoremlemma]{Lemma}
\newtheorem{proposition}[theoremproposition]{Proposition}

\theoremstyle{SpringerDefinition} %

\theoremstyle{SpringerExample} %

\newtheorem{remark}[theoremremark]{Remark}

\newcommand{\TheAuthor}{}
\newcommand{\Author}[1]%
        {\renewcommand{\TheAuthor}{#1}}                 
\newcommand{\TheRunningTitle}{}
\newcommand{\RunningTitle}[1]%
        {\renewcommand{\TheRunningTitle}{#1}}           

\usepackage{fancyhdr}                                   
\renewcommand{\footrulewidth}{0.4pt}

\usepackage{titling}
\pretitle{\begin{center}\LARGE\sffamily\slshape}        
\posttitle{\par\end{center}\vskip 0.5em}
\predate{\begin{center}\small}                          
\postdate{\par\end{center}\vskip 0.5em}
\thanksmarkseries{fnsymbol}                             

\usepackage[runin]{abstract}
\abslabeldelim{.\,}

\usepackage[titles]{tocloft}                            

\usepackage{titlesec}
\titleformat{\section}%
   {\Large\scshape\filcenter}{\thesection.}{1em}{}
\titleformat{\subsection}%
   {\large\sffamily\slshape\filcenter}{\thesubsection.}{1em}{}
\titleformat{\subsubsection}%
   {\normalsize\sffamily\filcenter}{\thesubsubsection.}{1em}{}
\titlespacing*{\section}%
   {0pt}{4.5ex plus 1ex minus .2ex}{3.3ex plus .2ex}
\titlespacing*{\subsection}%
   {0pt}{4.25ex plus 1ex minus .2ex}{3.05ex plus .2ex}
\titlespacing*{\subsubsection}%
   {0pt}{4.25ex plus 1ex minus .2ex}{3.05ex plus .2ex}

\usepackage[comma,numbers,square,sort&compress]{natbib} 

\usepackage{bm}                                         

\usepackage[
            ]{graphicx}                                 

\usepackage[
            small,%
            bf]{caption}                     
\setcaptionmargin{0.5in}

\newenvironment{acknowledgment}%
  {\begin{trivlist}\item[]\textbf{Acknowledgments.}}{\end{trivlist}}

\makeatletter
\providecommand*{\toclevel@assumption}{0}
\providecommand*{\toclevel@conjecture}{0}
\providecommand*{\toclevel@corollary}{0}
\providecommand*{\toclevel@definition}{0}
\providecommand*{\toclevel@example}{0}
\providecommand*{\toclevel@figure}{0}
\providecommand*{\toclevel@hypothesis}{0}
\providecommand*{\toclevel@lemma}{0}
\providecommand*{\toclevel@proof}{0}
\providecommand*{\toclevel@proposition}{0}
\providecommand*{\toclevel@remark}{0}
\providecommand*{\toclevel@theorem}{0}
\makeatother

\newcommand{\eref}[1]{\hyperref[{#1}]{(\ref*{#1})}}

\newcommand{\C}{\mathbb{C}}

\newcommand{\R}{\mathbb{R}}
\newcommand{\Z}{\mathbb{Z}}

\def\det{\mathop\mathrm{det}\nolimits}

\def\dim{\mathop\mathrm{dim}\nolimits}

\def\exp{\mathop\mathrm{exp}\nolimits}

\def\gker{\mathop\mathrm{gker}\nolimits}
\def\id{\mathcal{I}}
\def\ker{\mathop\mathrm{ker}\nolimits}

\def\Span{\mathop\mathrm{span}\nolimits}

\def\Re{\mathop\mathrm{Re}\nolimits}

\def\coloneqq{\mathrel{\mathop:}=}

\newcommand{\rma}{\mathrm{a}}

\newcommand{\rmc}{\mathrm{c}}
\newcommand{\rmd}{\mathrm{d}}
\newcommand{\rme}{\mathrm{e}}

\newcommand{\rmi}{\mathrm{i}}

\newcommand{\rmn}{\mathrm{n}}

\newcommand{\rmr}{\mathrm{r}}

\newcommand{\rmz}{\mathrm{z}}

\newcommand{\rmT}{\mathrm{T}}

\newcommand{\calA}{\mathcal{A}}
\newcommand{\calB}{\mathcal{B}}
\newcommand{\calC}{\mathcal{C}}

\newcommand{\calH}{\mathcal{H}}

\newcommand{\calJ}{\mathcal{J}}

\newcommand{\calL}{\mathcal{L}}

\newcommand{\calP}{\mathcal{P}}
\newcommand{\calQ}{\mathcal{Q}}

\newcommand{\calS}{\mathcal{S}}
\newcommand{\calT}{\mathcal{T}}

\newcommand{\vD}{\bm{\mathit{D}}}

\newcommand{\vK}{\bm{\mathit{K}}}

\newcommand{\vS}{\bm{\mathit{S}}}
\newcommand{\vU}{\bm{\mathit{U}}}

\newcommand{\vc}{\bm{\mathit{c}}}

\newcommand{\vx}{\bm{\mathit{x}}}
\newcommand{\vu}{\bm{\mathit{u}}}
\newcommand{\vv}{\bm{\mathit{v}}}
\newcommand{\vw}{\bm{\mathit{w}}}

\newcommand{\vn}{\bm{\mathit{0}}}

\newcommand{\vomega}{\boldsymbol{\omega}}

\newcommand{\gl}{\mathfrak{g}}

\Author{J. Bronski, M. Johnson, and T. Kapitula}
\RunningTitle{Instability index theory for quadratic pencils}

\usepackage[letterpaper,%
            bookmarks=false,%
            bookmarksopen=false,%
            bookmarksnumbered=false,%
            hypertexnames=false,%
            colorlinks=true,%
            linkcolor=blue,%
            citecolor=blue,%
            urlcolor=blue]{hyperref}                    

\usepackage{kpfonts}

\begin{document}

\title{An instability index theory for quadratic pencils and applications}
\author{%
        Jared Bronski %
        \thanks{E-mail: \href{mailto:jared@math.uiuc.edu}{jared@math.uiuc.edu}} \\
        Department of Mathematics \\
        University of Illinois Urbana-Champaign \\
        Urbana, IL 61801 %
\and
        Mathew A. Johnson %
        \thanks{E-mail: \href{mailto:matjohn@math.ku.edu}{matjohn@math.ku.edu}} \\
        Department of Mathematics \\
        University of Kansas \\
        Lawrence, KS 66045%
\and
        Todd Kapitula %
        \thanks{E-mail: \href{mailto:tmk5@calvin.edu}{tmk5@calvin.edu}} \\
        Department of Mathematics and Statistics \\
        Calvin College \\
        Grand Rapids, MI 49546 %
}


\begin{titlingpage}
\usethanksrule
\setcounter{page}{0}                                    
\maketitle
\begin{abstract}
Primarily motivated by the stability analysis of nonlinear waves in
second-order in time Hamiltonian systems, in this paper we develop an
instability index theory for quadratic operator pencils acting on a Hilbert
space. In an extension of the known theory for linear pencils, explicit
connections are made between the number of eigenvalues of a given quadratic
operator pencil with positive real parts to spectral information about the
individual operators comprising the coefficients of the spectral parameter in
the pencil.  As an application, we apply the general theory developed here to
yield spectral and nonlinear stability/instability results for abstract
second-order in time wave equations. More specifically, we consider the
problem of the existence and stability of spatially periodic waves for the
``good" Boussinesq equation. In the analysis our instability index theory
provides an explicit, and somewhat surprising, connection between the
stability of a given periodic traveling wave solution of the ``good"
Boussinesq equation and the stability of the same periodic profile, but with
different wavespeed, in the nonlinear dynamics of a related generalized
Korteweg-de Vries equation.
\end{abstract}

\cancelthanksrule
\renewcommand{\footrulewidth}{0.0pt}                    

\pdfbookmark[1]{\contentsname}{toc}                     
\tableofcontents                                        
\end{titlingpage}

\section{Introduction}\label{sec:intro}

When analyzing equations arising in mathematical physics and engineering, the question of stability
of special families of solutions is of prominent importance as it generally determines those solutions which
are most likely to be observed in physical applications.  In particular, solutions which are unstable do not naturally,
i.e. in absence of a controller, arise in applications except possibly as transient phenomena.  Furthermore,
stability analysis is often the first step in the study of finer phenomena such as transient behavior,
bifurcation, and the ability to control a wave to restrict it to a stable configuration.  In this paper, we are primarily
motivated by recent studies into the stability of traveling wave solutions of second-order in time
Hamiltonian equations of the form
\begin{equation}\label{e:ham}
\partial_t^2u+\mathcal{L}_xu+\mathcal{N}(u)=0,\quad (t,x)\in\mathbb{R}^2
\end{equation}
where $\mathcal{L}_x$ is a self-adjoint linear operator acting on the
$x$-variable only, and $\mathcal{N}(u)$ denotes nonlinear terms (e.g., see
\citep{angulo,arruda:nsp09,hakkaev:lsa12,stanis}).  A fundamental
characteristic of such PDE is that they take into account weak effects of
both nonlinearity and dispersion, and they arise naturally, for instance, as
models for propagation of waves in nonlinear strings and in the study of
bi-directional water wave propagation in the small amplitude, long wavelength
regime.

In regards to the latter water wave application, an equation of particular
interest in this paper (see \autoref{s:gbou} below) is the generalized
``good" Boussinesq (gB) equation
\begin{equation}\label{e:gbou}
\partial_t^2u-\partial_x^2\left(\partial_x^2u-u+f(u)\right)=0,
\end{equation}
which is a variant of one of one of the equations formulated by Boussinesq in
the 1870's in precisely this physical context. While the nonlinear stability
and instability of solitary waves in \eref{e:gbou} is by now well understood
(see \citep{bona:geo88,liu}), the stability (whether linear or nonlinear) of
the periodic traveling waves have received considerably less attention, and
results only exist for very special classes of solutions; namely, those
expressible in terms of Jacobi-elliptic functions.  One of the main
applications of the theory developed in this paper is an explicit connection
between periodic traveling waves of \eref{e:gbou} and the stability of the
same traveling wave profile (but with a \emph{different} wavespeed) in the
nonlinear dynamics governed by the generalized Korteweg-de Vries (gKdV)
equation
\begin{equation}\label{e:gkdv}
\partial_tu+\partial_x\left(\partial_x^2u+f(u)\right)=0.
\end{equation}
We consider this observation as a major contribution to the theory of
traveling periodic waves in \eref{e:gbou}. Indeed, since the stability of
periodic traveling waves in gKdV equations has been under intense
investigation over the last few years (see, for instance,
\citep{angulobona,arruda:nsp09,brj,bronski:ait11,deconinck:ots10,deconinck:tos10,Jkdv}),
this connection between the dynamics of gB and gKdV near periodic traveling
waves allows one to immediately translate known results for the stability of
gKdV periodic waves to results about periodic waves in gB. Applications of
this connection will be illustrated in \autoref{s:case1} and
\autoref{s:case2} below.

Suppose that $u(x-ct)$ is a traveling wave solution to \eref{e:ham}. When
considering small perturbations to this wave of the form $e^{\lambda
t}v(x-ct),\,\lambda\in\C$, we are naturally led to quadratic spectral
problems of the form
\[
\lambda^2 v -2c\lambda\partial_xv+\left(c^2\partial_x^2+\mathcal{L}_x+\mathcal{N}'(u)\right)v=0.
\]
More generally, when considering the spectral stability of the given
nonlinear dispersive wave $u$ it becomes important to understand the spectrum
of quadratic operator pencils of the form
\begin{equation}\label{e:i1}
\calP_2(\lambda)\coloneqq\calA+\lambda\calB+\lambda^2\calC,
\end{equation}
where the operators $\calA$ and $\calC$ are self-adjoint on some Hilbert
space $X$, endowed with an inner-product $\langle\cdot,\cdot\rangle$, and
$\calB$ is skew-symmetric
on X.
The class of perturbations considered in our applications naturally leads us
to assume that
the domain of each of the operators $\calA,\calB,\calC$
is dense in $X$ 
and that they each enjoy a particular compactness property
which guarantees that the spectrum of $\calP_2(\lambda)$, i.e., the collection of values $\lambda$ for which
$\calP_2(\lambda)$ fails to be boundedly invertible, is composed of point spectrum
only, that each eigenvalue has finite algebraic multiplicity, and the only
accumulation point of the eigenvalues is infinity (see
\autoref{l:noessential} below for a precise statement).
In this paper
$\sigma(\calP_2)$ will denote the collection of all eigenvalues for the
pencil.

Due to its clear connection and importance in analyzing the spectral
stability of traveling wave solutions of equations of the form \eref{e:ham},
our main theoretical results concern extending many previously known results
regarding the number of eigenvalues of $\calP_2$ with positive real part. For
the readers convenience, we now briefly recall the relevant known results.
The spectrum was studied in \citep{pivovarchik:oso07} under the assumption
that $\calA$ has compact resolvent, and $\calB,\calC$ are bounded and
positive semi-definite. The operators $\calB,\calC$ are not assumed to have
any symmetry properties, however. While it is not shown in that paper, by
\autoref{l:noessential} it is then known that the pencil only possesses point
eigenvalues, each of which has finite algebraic multiplicity. It is shown in
that paper that if $\calA$ is positive semi-definite, then all of the
spectrum is located in the closed left-half of the complex plane. If $\calA$
is not definite, it is shown that if $\calC=\id$ and $\calB$ is positive
definite, then the total number of eigenvalues in the closed right-half of
the complex-plane is equal to the number of negative eigenvalues of $\calA$.
As is seen in, e.g., \autoref{s:3}, it is not necessary that $\calC=\id$ in
order for that result to hold; indeed, all that is needed is that $\calC$ be
positive-definite. In \citep{lyong:tsp93} similar results are shown regarding
the number of eigenvalues with positive real part under the assumption that
$\calC=\id$ and $\calB$ is positive semi-definite. The spectrum of operators
of the form $\calP_2$ was studied in \citep{gurski:sdo04} under the
assumptions that both $\calA,\,\calC$ are positive definite and self-adjoint,
while $\calB$ is a negative definite self-adjoint operator (also see
\citep{kollar:hmf11} for a generalization). Under assumptions different than
those given in \autoref{l:noessential} (in particular, both $\calA$ and
$\calC$ are assumed to be compact) it is shown therein that in
$\sigma(\calP_2)$ there is an infinite number of positive eigenvalues which
have zero as a limit point. Finally, the matrix-valued version of the pencil
was studied in \citep{kollar:hmf11,chugunova:oqe09} under the assumption that
$\calC=\id$ and $\calA,\,\calB$ self-adjoint. Therein a parity index is given
relating the number of eigenvalues with positive real part to the number of
negative eigenvalues of $\calA$. This result strongly depends upon the fact
that the operators are matrix-valued, and consequently the Hilbert space
under question is finite-dimensional.

One of the main goals of this paper is thus to extend the previous theory
regarding the number of eigenvalues (counting multiplicity) with positive
real part; see Section \ref{s:3} below. Throughout this analysis, two
underlying assumptions will be:
\begin{enumerate}
\item $\rmn(\calA),\rmn(\calC)<+\infty$, where $\rmn(\calS)$ refers to the number of negative
eigenvalues (counting multiplicity) of the self-adjoint operator $\calS$
\item $\calC$ is invertible, i.e., $\rmz(\calC)=0$, where $\rmz(\calS)=\dim[\ker(\calS)]$ for
the self-adjoint operator $\calS$.
\end{enumerate}
Under these assumptions we have the intuition that the sum
$\rmn(\calA)+\rmn(\calC)$ acts as an upper bound on the total number of
eigenvalues of $\calP_2$ with positive real part. In order to see this,
consider the following. If the linear term in the pencil is dropped, then the
quadratic pencil is equivalent to the linear pencil
\[
\left[\left(\begin{array}{cc}\calA&0\\0&\calC^{-1}\end{array}\right)-\lambda
\left(\begin{array}{rr}0&\id\\-\id&0\end{array}\right)\right]
\left(\begin{array}{c}u\\v\end{array}\right)=
\left(\begin{array}{c}0\\0\end{array}\right).
\]
Note that the first term in the pencil is self-adjoint, whereas the second
term is skew-symmetric. Linear pencils of this form have been well-studied
(e.g., see
\citep{haragus:ots08,kapitula:cev04,kapitula:ace05,pelinovsky:ilf05} and the
references therein), and for this problem it has firmly been established that
the intuition is indeed correct. The technical difficulty is then the
inclusion of the linear term, and it is overcome in the establishment of
\autoref{thm:index} below.
In particular, in the case
where $\calP_2(\lambda)$ arises in the stability analysis of a given periodic traveling wave (as described previously),
a sufficient condition for spectral stability is that $\rmn(\calA)+\rmn(\calC)=0$.

In order to achieve an equality relating the number of eigenvalues of
$\calP_2$ with positive real part to spectral properties of the operators
$\calA$, $\calB$, and $\calC$ themselves, two additional factors must be
taken into account:
\begin{enumerate}
\item the effect of $\calA$ having a nontrivial kernel
\item purely imaginary eigenvalues having a negative Krein index (see \autoref{s:3}).
\end{enumerate}
The first factor arises because in applications the presence of symmetries
yields the existence of a nontrivial kernel, while the consideration of the
second factor is necessary in order to remove the intuitive inequality and
make it an equality.

The paper is organized as follows. In \autoref{s:2} it is shown that under
general assumptions, which are natural for the applications we have in mind,
there is no essential spectrum for the quadratic operator pencil $\calP_2$;
in other words, there will be only point eigenvalues, and each eigenvalue
will have finite algebraic multiplicity. The proof of this result easily
generalizes to polynomial operators of arbitrary order; see
\autoref{r:higherorderpencil} below. In \autoref{s:3} the main theoretical
result of the paper relating the number of eigenvalues of $\calP_2$ with
positive real part to the spectral properties of $\calA$, $\calB,$ and
$\calC$, are stated and proved.  Finally, in \autoref{s:4} we apply the
theoretical results developed in \autoref{s:2} and \autoref{s:3} to the study
of the spectral and orbital stability of nonlinear waves to second-order in
time Hamiltonian systems.  The first application develops a general
theoretical result concerning the stability of steady states in abstract
nonlinear wave equations; see \autoref{s:wave}. The second application
examines the spectral and nonlinear (orbital) stability of periodic traveling
waves in the ``good" Boussinesq equation \eref{e:gbou}; see \autoref{s:gbou}.
As described above, of particular interest in our study of the Boussinesq
equation is that we establish in \autoref{lem:51} and
\autoref{thm:orbstable2} a rigorous connection between a given stationary
periodic wave $u$ of the equation
\[
\partial_t^2u-2c\partial_{tx}^2u+\partial_x^2\left(\partial_{x}^2u-(1-c^2)u+f(u)\right)=0,\quad |c|<1,
\]
corresponding to the traveling wave $u(x-ct)$ in the gB, and the stability of
the same stationary wave profile in the gKdV,
\[
\partial_tu+\partial_x\left(\partial_x^2-\sqrt{1-c^2}\,u+f(u)\right)=0.
\]
To illustrate the power of this connection, in \autoref{s:case1} we translate
recent results by Bronski et al. \citep{bronski:ait11} concerning the
stability of periodic KdV waves with power-law nonlinearity to results for
the corresponding Boussinesq equation, allowing one to see the complete
stability picture of periodic waves in gB in the cases considered. In
\autoref{s:case2}, using known asymptotic analyses worked out in the context
of stability theory for gKdV waves (see \citep{brj}), we analyze the
stability of periodic waves with power-law nonlinearity which are near (in an
appropriate sense) either the solitary wave or equilibrium solutions of gB.
The analysis near the solitary wave makes a beautiful connection with the
known nonlinear stability/instability results for the nearby solitary waves.
Similarly, the analysis near the equilibrium solution provides new insights
into the transitions to instability in the periodic context; in particular,
we find that for a given power-law nonlinearity stable periodic traveling
wave solutions always exist, even when no stable solitary wave exists.

\begin{acknowledgment}
TK gratefully acknowledges the support of the Jack and Lois Kuipers Applied
Mathematics Endowment, a Calvin Research Fellowship, and the National Science
Foundation under grants DMS-0806636 and DMS-1108783. MJ was partially
supported by the University of Kansas General Research Fund allocation
2302278.  JB gratefully acknowledges support from the National Science
Foundation under grant DMS-DMS-0807584.
\end{acknowledgment}

\section{Preliminary result: spectra is point only}\label{s:2}

\noindent\textbf{Notation:} In this paper, and particularly in this and the
subsequent section, the notion of \textsl{matrix representation} of a
self-adjoint operator $\calS$ constrained to a subspace $E$, i.e., $S|_{E}$,
will often be used. Let $\{e_1,\dots,e_n\}$ be a basis for $E$, and let
$P_E:X\mapsto E$ be the orthogonal projection. If the basis is orthonormal,
then one can write
\[
P_E=\sum_{j=1}^n\langle\cdot,e_j\rangle e_j.
\]
Setting $S|_E=P_E\calS P_E:E\mapsto E$, the quadratic form
\[
\langle u,S|_Eu\rangle=\langle P_Eu,P_E\calS P_Eu\rangle=\vc\cdot\vS\vc,
\]
where $P_Eu=\sum c_je_j,\,\vc=(c_1,\dots,c_n)^\rmT$, and the symmetric matrix
$\vS\in\R^{n\times n}$ is given by $\vS_{ij}=\langle e_i,\calS e_j\rangle$.
The matrix $\vS$ is precisely the matrix representation for the self-adjoint
operator $\calS|_E$. In the rest of this paper the symbol $\calS|_E$ will be
used to represent the matrix representation $\vS$ of $\calS$ constrained to
operate on the subspace $E$.

\vspace{2mm}

The goal of this section is to demonstrate that for the quadratic pencil of
\eref{e:i1}, i.e.,
\[
\calP_2(\lambda)=\calA+\lambda\calB+\lambda^2\calC,
\]
there is point spectra only, that each eigenvalue has finite multiplicity,
and infinity is the only possible limit point of the eigenvalues. The result
requires the use of \citep[Theorem~12.9]{markus:itt88}, in which it is shown
that the spectrum has the desired properties for the polynomial operator
\[
\calP_n(\lambda)=\id+\sum_{j=1}^n\lambda^n\calA_j
\]
if each operator $\calA_j,\,j=1,\dots,n$, is compact. With this result in
mind, first assume that $\calA$ is invertible, and rewrite the eigenvalue
problem as
\[
\calA(\id+\lambda\calA^{-1}\calB+\lambda^2\calA^{-1}\calC)u=0.
\]
Of course, since $\calA$ is nonsingular this problem is equivalent to
\[
(\id+\lambda\calA^{-1}\calB+\lambda^2\calA^{-1}\calC)u=0.
\]
If the operators $\calA^{-1}\calB$ and $\calA^{-1}\calC$ are both compact,
then the result immediately follows from \citep[Theorem~12.9]{markus:itt88}.

On the other hand, now suppose that $\ker(\calA)$ is nontrivial, but that
$\calA$ has compact resolvent. The proof of the desired result will now be
accomplished via the construction and evaluation of a generalized
\textsl{Krein matrix}, which was recently introduced in a general form in
\citep{kapitula:tks10}. Let $P_\calA:X\mapsto\ker(\calA)$ be the orthogonal
projection. Writing $u=a+a^\perp$, where $a\in\ker(\calA)$ and $a^\perp\in
\ker(\calA)^\perp$, the eigenvalue problem becomes
\begin{equation}\label{e:pp2}
\calP_2(\lambda)a+\calP_2(\lambda)a^\perp=0.
\end{equation}
Defining the complementary projection $P_{\calA}^\perp\coloneqq\id-P_\calA$,
applying this projection to \eref{e:pp2}, and solving for
$a^\perp=P_{\calA}^\perp a^\perp$ yields
\begin{equation}\label{e:pp3}
a^\perp=-(P_{\calA}^\perp\calP_2(\lambda)P_{\calA}^\perp)^{-1}P_{\calA}^\perp\calP_2(\lambda)a.
\end{equation}
In the formulation of \eref{e:pp3} it is implicitly being assumed that
$\calP_2(\lambda)a^\perp\neq0$. If $\calP_2(\lambda)a^\perp=0$, then
$\lambda$ is an eigenvalue whose eigenfunction is in $\ker(\calA)^\perp$:
this follows immediately from \eref{e:pp2} upon setting $a=0$. Since
$\calP_2(\lambda)a=0$, the potential pole singularity for such a $\lambda$ is
removable. If the inner-product with $a$ is now taken in \eref{e:pp2}, then
it is seen that
\begin{equation}\label{e:pp4}
\langle a,\calP_2(\lambda)a\rangle+\langle a^\perp,\calP_2(\lambda)^\rma a\rangle=0,
\end{equation}
where
\[
\calP_2(\lambda)^\rma=\calA-\overline{\lambda}\calB+\overline{\lambda}^2\calC
\]
is the adjoint operator for the original pencil. Note that the fact that
$\calA,\calC$ are self-adjoint and $\calB$ is skew-symmetric was used in this
formulation of the adjoint pencil. Substituting the expression for $s^\perp$
given in \eref{e:pp3} into \eref{e:pp4} yields the linear system
\begin{equation}\label{e:pp5}
\vK_2(\lambda)\vx=\vn,\quad
\vK_2(\lambda)\coloneqq\calP_2(\lambda)|_{\ker(\calA)}-%
(P_{\calA}^\perp\calP_2(\lambda)P_{\calA}^\perp)^{-1}|_{P_{\calA}^\perp\calP_2(\lambda)[\ker(\calA)]}.
\end{equation}
Here
\[
\vx\in\C^{\rmz(\calA)},\quad\vK_2(\lambda)\in\C^{\rmz(\calA)\times\rmz(\calA)},
\]
and
\[
\calP_2(\lambda)[\ker(\calA)]\coloneqq\{\calP_2(\lambda)a:a\in\ker(\calA)\}.
\]
The matrix $\vK_2(\lambda)$ is known as the Krein matrix.

Eigenvalues for the pencil are found either via $\vx=\vn$, which means that
if $\lambda$ is an eigenvalue, then the associated eigenfunction satisfies
$u\in\ker(\calA)^\perp$, or $\det[\vK_2(\lambda)]=0$. The quest for
eigenvalues is now a search for the zeros of the determinant of the Krein
matrix. The poles of the Krein matrix are the eigenvalues of the operator
\[
P_{\calA}^\perp\calP_2(\lambda)P_{\calA}^\perp=
P_{\calA}^\perp\calA P_{\calA}^\perp+\lambda P_{\calA}^\perp\calB P_{\calA}^\perp+
\lambda^2P_{\calA}^\perp\calC P_{\calA}^\perp:\ker(\calA)^\perp\mapsto\ker(\calA)^\perp.
\]
Since $P_{\calA}^\perp\calA
P_{\calA}^\perp:\ker(\calA)^\perp\mapsto\ker(\calA)^\perp$ is invertible,
from the earlier argument in this section we know that the operator
$P_{\calA}^\perp\calP_2(\lambda)P_{\calA}^\perp$ has point spectra only, each
eigenvalue has finite multiplicity, and infinity is the only possible limit
point of the eigenvalues. Thus, one can say that
$\det[\vK_2(\lambda)]:\C\mapsto\C$ is meromorphic, each singularity is a pole
of finite order, and the only possible accumulation point of the poles is
infinity. Since $\det[\vK_2(\lambda)]$ is meromorphic, it is then known that
each of its zeros is of finite order, and that the only possible accumulation
point of the zeros is infinity. Finally, it was demonstrated in
\citep{kapitula:tks10} that for linear pencils the order of the zero of
$\det[\vK_2(\lambda)]$ is equal to the algebraic multiplicity of the
eigenvalue. The proof of that result easily carries over to quadratic
pencils, and will be left for the interested reader.

\begin{lemma}\label{l:noessential}
Let $P_\calA:X\mapsto\ker(\calA)$ be the orthogonal projection, and let
$P_\calA^\perp=\id-P_\calA$ be the complementary projection. Suppose that the
operators
\[
(P_\calA^\perp\calA P_\calA^\perp)^{-1}P_\calA^\perp\calB P_\calA^\perp,\,%
(P_\calA^\perp\calA P_\calA^\perp)^{-1}P_\calA^\perp\calC P_\calA^\perp:\ker(\calA)^\perp\mapsto\ker(\calA)^\perp
\]
are both compact. Then the spectrum of the quadratic pencil
$\calA+\lambda\calB+\lambda^2\calC$ is point spectra only. Furthermore, each
eigenvalue has finite multiplicity, and infinity is the only possible limit
point of the eigenvalues.
\end{lemma}

\begin{remark}\label{r:higherorderpencil}
Note that the proof of \autoref{l:noessential} did not require that any of
the operators have a symmetry property. Thus, this lemma can be thought of as
a general result about quadratic pencils. Indeed, although it will not be
proven here, the above argument can be extended to show that for
$n^\mathrm{th}$-order polynomial pencils of the form
\[
\calP_n(\lambda)=\sum_{j=0}^n\lambda^j\calA_j,
\]
if each of the operators
$(P_{\calA_0}^\perp\calA_0P_{\calA_0}^\perp)^{-1}P_{\calA_0}^\perp\calA_jP_{\calA_0}^\perp$
is compact for $j=1,\dots,n$, then the spectrum for this pencil will have
exactly the same properties as that for the quadratic pencil.
\end{remark}

\section{Main result: the instability index theorem}\label{s:3}

The goal here is to derive an instability index theorem for the quadratic
pencil.
In applications it may be the case that $\calC^{-1}$ is bounded and not
compact, and/or $\calB$ is not bounded. In order to overcome these technical
difficulties (which are associated with the proof only), take a positive
definite self-adjoint operator $\calS$ which has a compact inverse, and
consider the quadratic pencil
\[
\widehat{\calP}_2(\lambda)\coloneqq\widehat{\calA}+\lambda\widehat{\calB}+%
\lambda^2\widehat{\calC}.
\]
The operators in this new pencil are related to those of the original pencil
by
\[
\widehat{\calA}\coloneqq\calS^{-1}\calA\calS^{-1},\quad%
\widehat{\calB}\coloneqq\calS^{-1}\calB\calS^{-1},\quad%
\widehat{\calC}\coloneqq\calS^{-1}\calC\calS^{-1},
\]
so that
\[
\calP_2(\lambda)=\calS\widehat{\calP}_2(\lambda)\calS.
\]
The working assumption is that $\calB,\,\calC$ are $\calS$-compact, while
$\calA$ has $\calS$-compact resolvent, where an operator $\calT$ is said to
be $\calS$-compact if the operator $\calS^{-1}\calT\calS^{-1}$ is compact,
and is said to have $\calS$-compact resolvent if the inverse
$\calS\calT^{-1}\calS$ (defined to act on the range space) is compact.

Define the space $X_\calS$ by
\[
X_\calS\coloneqq\{u\in X:\langle\calS u,\calS u\rangle<\infty\}.
\]
It will be assumed that $X_\calS\subset X$ is dense. Since $\calS^{-1}$ is
compact, and therefore bounded, it is clear that
$\sigma(\widehat{\calP}_2)\subset\sigma(\calP_2)$ on $X$. It is also clear
that $\sigma(\calP_2)$ when considered on the space $X_\calS$ is a subset of
$\sigma(\widehat{\calP}_2)$ when considered on the space $X$. In other words,
it is true that
\[
\sigma(\calP_2)\,\,\mathrm{on}\,\,X_\calS\subset%
\sigma(\widehat{\calP}_2)\,\,\mathrm{on}\,\,X\subset%
\sigma(\calP_2)\,\,\mathrm{on}\,\,X.
\]
Since $X_\calS\subset X$ is dense, by
\citep[Proposition~3.2]{shkalikov:opa96} it will be the case that
\[
\sigma(\calP_2)\,\,\mathrm{on}\,\,X_\calS=%
\sigma(\widehat{\calP}_2)\,\,\mathrm{on}\,\,X.
\]
Finally, since $\calS^{-1}$ is bounded, it is true that
$\rmn(\calA)=\rmn(\widehat{\calA})$ and $\rmn(\calC)=\rmn(\widehat{\calC})$.
In conclusion, from this point forward the ``hat"'s associated with the
operators can be dropped, and it will be assumed that the operators satisfy
$\calA^{-1},\,\calC,\,\calB$ are compact.

The instability index theorem for the quadratic pencil will be proven by
constructing an equivalent linear pencil, and then deriving an index theorem
for that linear pencil. Upon setting $w=(u,\lambda\calC u)^\rmT$, the
quadratic pencil \eref{e:i1} is linearized to become
\begin{equation}\label{e:h1}
(\calL-\lambda\calJ^{-1})w=0,
\end{equation}
where
\[
\calL=\left(\begin{array}{cc}\calA&0\\0&\calC^{-1}\end{array}\right),\quad%
\calJ=\left(\begin{array}{rr}0&\id\\-\id&-\calB\end{array}\right).
\]
Here is where the assumption that $\calC$ be invertible comes into play.
Since $\calB$ is skew-symmetric, \eref{e:h1}, and consequently the pencil
\eref{e:i1}, is formally equivalent to the Hamiltonian eigenvalue problem
\begin{equation}\label{e:h2}
\calJ\calL w=\lambda w,
\end{equation}
where $\calJ$ is skew-symmetric and $\calL$ is self-adjoint. Note that in
this formulation
\begin{enumerate}
\item $\calL$ has compact resolvent
\item $\calJ$ has bounded inverse
\item the spectrum satisfies the symmetry
$\{\lambda,-\overline{\lambda}\}\subset\sigma(\calJ\calL)$, and if all of
the operators have zero imaginary part, the symmetry becomes
$\{\pm\lambda,\pm\overline{\lambda}\}\subset\sigma(\calJ\calL)$.
\end{enumerate}

Since $\calC$ is invertible, it is clear that for nonzero $\lambda$ the two
eigenvalue problems have identical eigenvalues; furthermore, the geometric
multiplicities match. As it will be seen, it is also the case that the
algebraic multiplicities of these eigenvalues also coincide. We will show
that this is true at the origin: the proof will clearly generalize to the
case of a nonzero eigenvalue. We must consider the structure of
$\gker(\calJ\calL)$ and the manner in which it relates to
$\gker(\calP_2(0))$.

First consider $\gker(\calP_2(0))$. It is clear that
$\ker(\calP_2(0))=\ker(\calA)$. Following \citet{markus:itt88}, generalized
eigenfunctions are found by solving
\[
\calP_2(0)a_1+\calP_2'(0)a_0=0,\quad a_0\in\ker(\calA).
\]
This is equivalent to solving
\begin{equation}\label{e:h3}
\calA a_1=-\calB a_0,
\end{equation}
which by the Fredholm alternative has a nontrivial solution if and only if
$\calB a_0\in\ker(\calA)^\perp$. Since $\calB$ is skew-symmetric, it is not
unreasonable to assume that $\calB|_{\ker(\calA)}=\vn$. In this case there is
a solution to \eref{e:h3} for any $a_0\in\ker(\calA)$, and a full set of
associated eigenfunctions is given by $\calA^{-1}\calB\ker(\calA)$, where the
obvious notation
\[
\calA^{-1}\calB\ker(\calA)\coloneqq\{\calA^{-1}\calB a:a\in\ker(\calA)\}
\]
is being used. The next set of eigenfunctions is found by solving
\[
\calP_2(0)a_2+\calP_2'(0)a_1+\frac12\calP_2''(0)a_0=0,\quad%
a_0\in\ker(\calA),\,\,a_1=-\calA^{-1}\calB a_0\in\calA^{-1}\calB\ker(\calA).
\]
This equation is equivalent to
\begin{equation}\label{e:h4}
\calA a_2=-\calB a_1-\calC a_0=%
-(\calC-\calB\calA^{-1}\calB)a_0.
\end{equation}
Again using the Fredholm alternative, it is seen that there is a nontrivial
solution to \eref{e:h4} if and only if
$(\calB\calA^{-1}\calB-\calC)a_0\in\ker(\calA)^\perp$. Upon using the fact
that for any $a\in\ker(\calA)$,
\[
\langle\calB(\calA^{-1}\calB a_0),a\rangle=-\langle\calA^{-1}\calB a_0,\calB a\rangle,
\]
it is seen that if
\[
\vD\coloneqq%
(\calC-\calB\calA^{-1}\calB)|_{\ker(\calA)}
\]
is nonsingular, then $(\calC-\calB\calA^{-1}\calB)a_0\notin\ker(\calA)^\perp$
for any $a_0\in\ker(\calA)$. It will henceforth be assumed that $\vD$ is
nonsingular.

Now consider $\gker(\calJ\calL)$. It is clear that
$\ker(\calL)=(\ker(\calA),0)^\rmT$. 
Since
\begin{equation}\label{e:h4a}
\calJ^{-1}\ker(\calL)=%
\left\{\left(\begin{array}{c}-\calB a_0\\a_0\end{array}\right):a_0\in\ker(\calA)\right\},
\end{equation}
upon writing $w=(u,v)^\rmT$ the generalized eigenfunctions are found by
solving
\[
\calA u=-\calB a_0,\quad\calC^{-1}v=a_0.
\]
The first equation is precisely \eref{e:h3}, which was seen to have a
solution for any $a_0\in\ker(\calA)$. Consequently, there is a generalized
eigenspace at $\lambda=0$ which is given by $\{(-\calA^{-1}\calB a_0,\calC
a_0)^\rmT:a_0\in\ker(\calA)\}$. Since
\[
\calJ^{-1}\left(\begin{array}{c}-\calA^{-1}\calB a_0\\\calC a_0\end{array}\right)=%
-\left(\begin{array}{c}(\calC-\calB\calA^{-1}\calB)a_0\\\calA^{-1}\calB a_0\end{array}\right),
\]
the next set of generalized eigenfunctions is found by solving
\[
\calA u=-(\calC-\calB\calA^{-1}\calB)a_0,\quad%
\calC^{-1}v=-\calA^{-1}\calB a_0.
\]
The first equation is precisely \eref{e:h4}, which was seen to have no
solution under the assumption that $\vD$ is nonsingular.

It is now seen that regarding the algebraic multiplicity of the eigenvalue
the Hamiltonian linearization \eref{e:h2} is equivalent to the pencil at
$\lambda=0$. Following the same argument for nonzero $\lambda$ it is not
difficult to check that this equivalence continues to hold; namely, the
location of the eigenvalues, and their multiplicities, are the same for the
two systems.

We are now ready to derive the instability index for the linear Hamiltonian
system \eref{e:h2}. We must first construct the appropriate closed subspace
on which both $\calJ,\,\calL$ are nonsingular. Recalling that
$P_\calA^\perp:X\mapsto\ker(\calA)^\perp$ is the orthogonal projection, let
$\Pi_\calL^\perp:X\times X\mapsto\ker(\calA)^\perp\times X=\ker(\calL)^\perp$
be given by
\[
\Pi_\calL^\perp\coloneqq\left(\begin{array}{cc}P_\calA^\perp&0\\0&\id\end{array}\right);
\]
in other words, for $w=(u,v)\in X\times X$ it is true that $\Pi_\calL^\perp
w=(P_\calA^\perp u,v)$.
Define another orthogonal projection by
\begin{equation}\label{e:h5a}
\Pi_{\calJ^{-1}\ker(\calL)}^\perp:X\times X\mapsto[\calJ^{-1}\ker(\calL)]^\perp.
\end{equation}
Because $\ker(\calL)\perp\calJ^{-1}\ker(\calL)$, the projections
$\Pi_\calL^\perp,\,\Pi_{\calJ^{-1}\ker(\calL)}^\perp$ commute. Upon setting
$\Pi\coloneqq\Pi_\calL^\perp\Pi_{\calJ^{-1}\ker(\calL)}^\perp$ (note that
$\Pi$ being the composition of self-adjoint commuting operators implies that
it too is self-adjoint), nonzero eigenvalues for the linearization
\eref{e:h2} are found by solving
\begin{equation}\label{e:h6}
\Pi\calJ\Pi\cdot\Pi\calL\Pi\cdot\Pi w=\lambda\Pi w,\quad%
\Pi w\in[\ker(\calL)\oplus\calJ^{-1}\ker(\calL)]^\perp
\end{equation}
(e.g., see \citep[Section~2]{deconinck:ots10}). This is the eigenvalue
problem to be studied in the rest of this section.

The goal is to now count the total number of eigenvalues in the open
right-half of the complex plane (counting multiplicity), along with a those
eigenvalues on the imaginary axis which have negative Krein index. For each
$\lambda\in\rmi\R$ let $E_\lambda$ denote the generalized eigenspace. The
negative Krein index of the eigenvalue for the linearized system \eref{e:h2}
is given by
\[
k_\rmi^-(\lambda)\coloneqq\rmn(\calL|_{E_\lambda}),
\]
and if $k_\rmi^-(\lambda)=0$, then the eigenvalue is (often) said to have
positive Krein signature. Let us now relate this definition to a definition
using the quadratic pencil. By the definition of $w$ leading to the
linearization \eref{e:h1} one has that
\[
\calL|_{E_{\lambda}}=(\calA+|\lambda|^2\calC)|_{\Pi_1E_{\lambda}}=
(\calA-\lambda^2\calC)|_{\Pi_1E_{\lambda}},
\]
where $\Pi_1:X\times X\mapsto X$ is the projection onto the first component,
i.e., $\Pi_1(u,v)^\rmT=u$, and the second equality follows from the fact that
$\lambda\in\rmi\R$. Now, $\lambda$ being an eigenvalue with associated
eigenfunction $u$ means that
\[
\calP_2(\lambda)u=0\quad\Rightarrow\quad
\calA u=-\lambda\calB u-\lambda^2\calC u,
\]
which in turn implies
\[
(\calA-\lambda^2\calC)|_{\Pi_1E_{\lambda}}=%
-\lambda\calP_2'(\lambda)|_{\Pi_1E_{\lambda}}.
\]
In conclusion, the negative Krein index for the quadratic pencil is defined
to be
\begin{equation}\label{e:kindex}
k_\rmi^-(\lambda)\coloneqq\rmn(-\lambda\calP_2'(\lambda)|_{E_{\lambda}}),
\end{equation}
where now $E_\lambda=\gker(\calP_2(\lambda))$, i.e., the generalized
eigenspace of the quadratic pencil $\calP_2(\lambda)$ associated with the
eigenvalue $\lambda\in\rmi\R$.

We are now ready to derive the index formula. For the eigenvalue problem
\eref{e:h6} let $k_\rmr$ represent the number of positive real-valued
eigenvalues (counting multiplicity), and let $k_\rmc$ be the number of
complex-valued eigenvalues (counting multiplicity) with positive real part.
Furthermore, let the total negative Krein index be given by
\[
k_\rmi^-\coloneqq\sum_{\sigma(\calP_2(\lambda))\cap\rmi\R}k_\rmi^-(\lambda).
\]
Regarding the eigenvalue problem \eref{e:h6} it is known that
$(\Pi\calL\Pi)^{-1}$ is compact, and that the operator $\Pi\calJ\Pi$ is
bounded with bounded inverse. Indeed, since one can write
\[
\Pi\calJ\Pi=\Pi\left(\begin{array}{rr}0&\id\\-\id&0\end{array}\right)\Pi+
\Pi\left(\begin{array}{rr}0&0\\0&-\calB\end{array}\right)\Pi,
\]
one actually has that both $\Pi\calJ\Pi$ and $(\Pi\calJ\Pi)^{-1}$ can be
written as the sum of a bounded operator and a compact operator. Thus, upon
using the fact that compact operators are uniformly approximated by matrices,
when computing an index which takes into the account the (finite) number of
negative directions of an operator it is sufficient to consider the case of
matrices only. For the eigenvalue problem \eref{e:h6} when the operators are
matrices it is known from \citep{haragus:ots08} that
\begin{equation}\label{e:h9}
k_\rmr+k_\rmc+k_\rmi^-=\rmn(\Pi\calL\Pi).
\end{equation}
Since all of the quantities are integer-valued, by taking the limit one
deduces that the result holds for the full operators.

Before stating the final result, the quantity $\rmn(\Pi\calL\Pi)$ must be
computed in terms of the original operator $\calL$. It is known (e.g., see
\citep[Index Theorem]{kapitula:sif12}) that
\begin{equation}\label{e:h7}
\rmn(\Pi\calL\Pi)=\rmn(\calL)-\rmn(\calL^{-1}|_{\calJ^{-1}\ker(\calL)}).
\end{equation}
It is clearly the case that
\[
\rmn(\calL)=\rmn(\calA)+\rmn(\calC).
\]
Furthermore, it is straightforward to verify that
\[
\calL^{-1}|_{\calJ^{-1}\ker(\calL)}=
(\calC-\calB\calA^{-1}\calB)|_{\ker(\calA)}.
\]
The instability index \eref{e:h9} can then be rewritten as
\begin{equation}\label{e:h8}
\rmn(\Pi\calL\Pi)=\rmn(\calA)+\rmn(\calC)-\rmn((\calC-\calB\calA^{-1}\calB)|_{\ker(\calA)}).
\end{equation}
Combining \eref{e:h9} with \eref{e:h8} yields the following theorem:

\begin{theorem}\label{thm:index}
Suppose that the operators satisfy the assumption of \autoref{l:noessential}.
Further suppose that there is a self-adjoint and positive operator $\calS$
such that
\begin{enumerate}
\item $X_\calS\coloneqq\{u\in X:\langle\calS u,\calS u\rangle<\infty\}\subset
X$ is dense
\item the operators $\calB,\calC$ are $\calS$-compact, and the operator
$\calA$ has an $\calS$-compact resolvent.
\end{enumerate}
Finally, assume that the operator $\calC$ is invertible. If
\begin{enumerate}
\item[(i)] $\calB|_{\ker(\calA)}=\vn_{\rmz(\calA)}$, and
\item[(ii)] $(\calC-\calB\calA^{-1}\calB)|_{\ker(\calA)}$ is invertible,
\end{enumerate}
then the total number of eigenvalues in the closed right-half of the complex
plane satisfies the instability index
\[
k_\rmr+k_\rmc+k_\rmi^-=\rmn(\calA)+\rmn(\calC)-%
\rmn((\calC-\calB\calA^{-1}\calB)|_{\ker(\calA)}).
\]
\end{theorem}

\begin{remark}
If $\calA$ is nonsingular, then \autoref{thm:index} is precisely the result
of \citep[Corollary~3.9]{shkalikov:opa96}. On the other hand, in the event
that $\calA$ has a nontrivial kernel then it is an improvement over
\citep[Theorem~4.2]{shkalikov:opa96}, where it was shown that
\[
k_\rmr+k_\rmc+k_\rmi^-\le\rmn(\calA)+\rmn(\calC).
\]
The proof presented here is quite different than that given in
\citep{shkalikov:opa96}; in particular, in that paper the analysis takes
place on Pontryagin spaces, and the linearization that is studied there is
not the Hamiltonian linearization of \eref{e:h2}.
\end{remark}

\begin{remark}
If the imaginary part of all of the operators is zero, then due to the
Hamiltonian eigenvalue symmetry
$\{\pm\lambda,\pm\overline{\lambda}\}\subset\sigma(\calP_2)$ it is
necessarily the case that $k_\rmc$ and $k_\rmi^-$ are even-value. Thus, under
this assumption $\rmn(\Pi\calL\Pi)$ being odd automatically implies that
$k_\rmr\ge1$. On the other hand, if one or more of the operators has a
nontrivial imaginary part, then the Hamiltonian eigenvalue symmetry reduces
to $\{\lambda,-\overline{\lambda}\}\subset\sigma(\calP_2)$, and no such
conclusion can be drawn.
\end{remark}

\begin{remark}
One consequence of the index is that all but a finite number of eigenvalues
are purely imaginary; furthermore, the purely imaginary eigenvalues have
positive Krein signature if the modulus is sufficiently large.
\end{remark}

\begin{remark}
The proof of the index formula \eref{e:h9} in \citep{haragus:ots08} first
required the SCS Basis Lemma, in which it was shown that the generalized
eigenvectors associated with the linearization \eref{e:h1} formed a basis.
Technical assumptions on the operators were needed in order for the SCS Basis
Lemma to hold true. Unfortunately, at least in the application discussed
later in this paper these technical assumptions do not hold; hence, the
alternate proof via the limiting argument.
\end{remark}

\section{Applications to second-order in time Hamiltonian systems}\label{s:4}

As discussed in the introduction, quadratic operator pencils arise naturally
in when one studies the stability of solutions to second-order in time
Hamiltonian systems.  In this section, we present two applications of the
general theory developed in the previous sections in precisely this context.
We begin by considering the stability of periodic waves in an (abstract)
nonlinear wave equation posed in a Hilbert space, and conclude with a
stability analysis for periodic waves in the so-called ``good" Boussinesq
equation.

\subsection{Example: stability in (abstract) nonlinear wave equations}\label{s:wave}
One important example of quadratic pencils arises in the study of
second-order (in time) Hamiltonian systems (for a specific case of the
following discussion, see, e.g., \citep[Section~7]{grillakis:sto90}).
Consider a wave equation of the form
\begin{equation}\label{e:sh1}
\partial_t^2u+\calH'(u)=0,\quad%
\calH^{(k)}(u)\coloneqq\frac{\delta^k\calH}{\delta u^k}(u),
\end{equation}
where $u\in X$, which is a Hilbert space with inner-product
$\langle\cdot,\cdot\rangle$. The Hamiltonian $\calH:X\mapsto\R$ is assumed to
be smooth.

It will be assumed that the Hamiltonian system has symmetries. Let $G$ be a
finite-dimensional abelian Lie group with Lie algebra $\gl$. Denote by
$\exp(\omega)=\rme^\omega$ for $\omega\in\gl$ the exponential map from $\gl$
into $G$, and assume that $\calT:G\mapsto L(V)$, where $X\subset V\subset
X^*$ (the dual space of $X$), is a unitary representation of $G$ on $V$. It
is then the case that $\calT'(e)$ maps $\gl$ into the space of closed
skew-symmetric operators on $V$ with domain $X$. The notation
$\calT_\omega\coloneqq\calT'(e)\omega$ for $\omega\in\gl$ will be used to
denote the linear skew-symmetric operator which is the generator of the
semigroup $\calT(\rme^{\omega t})$. Using this notation the symmetry
assumption becomes that the Hamiltonian satisfies
$\calT(\omega)\calH(u)=\calH(\calT(\omega)u)$ for all $\omega\in\gl$.

Writing $\vu=(u,v)^\rmT$, where $v=\partial_tu\in X_1$ (in applications, it
is often the case that $X\subset X_1$ is dense), the system \eref{e:sh1} can
be written on $X\times X_1$ as the first-order Hamiltonian system
\begin{equation}\label{e:sh2}
\partial_t\vu=\calJ\widehat{\calH}'(\vu),
\end{equation}
where
\[
\calJ=\left(\begin{array}{rr}0&\id\\-\id&0\end{array}\right),\quad
\widehat{\calH}(\vu)=\calH(u)+\frac12\langle v,v\rangle.
\]
The system \eref{e:sh2} is invariant under the action
$\widehat{\calT}(\vomega)$, where
\[
\widehat{\calT}(\vomega)\vu=%
\left(\begin{array}{rr}\calT(\vomega)&0\\0&\calT(\vomega)\end{array}\right)\vu.
\]
An $n$-parameter family of conserved quantities for the Hamiltonian system
\eref{e:sh2} is induced from the self-adjoint operator
$\calJ^{-1}\widehat{\calT}_\omega$, and is given by
\[
\calQ(\vu)\coloneqq\frac12\langle\calJ^{-1}\widehat{\calT}_\omega\vu,\vu\rangle=%
-\Re\left(\langle\calT_\omega u,v\rangle\right).
\]

Upon defining the Lagrangian
\[
\Lambda(\vu)\coloneqq\widehat{\calH}(\vu)+\calQ(\vu),
\]
waves to \eref{e:sh2} will be realized as steady-state solutions for the
system
\begin{equation}\label{e:sh3}
\partial_t\vu=\calJ\Lambda'(\vu),
\end{equation}
i.e., they are critical points for the Lagrangian. Since
\[
\Lambda'(\vu)=\widehat{\calH}'(\vu)+%
\calJ^{-1}\widehat{\calT}_\omega(\vu)=%
\left(\begin{array}{c}\calH'(u)+\calT_\omega v\\v-\calT_\omega u\end{array}\right),
\]
critical points are solutions to
\begin{equation}\label{e:sh4}
\calH'(u)+\calT_\omega^2u=0,\quad%
\calT_\omega^2u\coloneqq\calT_\omega(\calT_\omega u).
\end{equation}
It should be noted here that \eref{e:sh3} is equivalent to the second-order
problem
\begin{equation}\label{e:sh4a}
\partial_t^2u+2\calT_\omega\partial_tu+\calH'(u)+\calT_\omega^2u=0.
\end{equation}


Suppose that $u=U$ is a solution to \eref{e:sh4} (the $\omega$-dependence of
the solution is being suppressed here), so that $\vU=(U,\calT_\omega U)^\rmT$
is a critical point of the Lagrangian. Indeed, further suppose that there is
a nonempty open set $\Omega\subset\gl$ such that the solution is smooth in
$\omega$ for all $\omega\in\Omega$, and further assume that the isotropy
subgroups $\{g\in G:\calT(g)U=U\}$ are discrete for all $\omega$. Now
consider the spectral and orbital stability of the wave. The linearized
problem associated with \eref{e:sh3} is given by
\begin{equation}\label{e:sh5}
\partial_t\vu=\calJ\calL\vu,
\end{equation}
where the self-adjoint operator $\calL$ is
\[
\calL\coloneqq\Lambda''(\vU)=%
\left(\begin{array}{cc}\calH''(U)&\calT_\omega\\-\calT_\omega&\id\end{array}\right).
\]
The eigenvalue problem for \eref{e:sh5} is given by
\[
\calJ\calL\vu=\lambda\vu.
\]
This eigenvalue problem is the system
\[
-\calH''(U)u-\calT_\omega v=\lambda v,\quad%
-\calT_\omega u+v=\lambda u,
\]
which after substitution is equivalent to the quadratic pencil
\begin{equation}\label{e:sh6}
(\calH''(U)+\calT_\omega^2+2\lambda\calT_\omega+\lambda^2\id)u=0.
\end{equation}
In the notation of \eref{e:i1} one has
\[
\calA=\calH''(\phi)+\calT_\omega^2,\quad%
\calB=2\calT_\omega,\quad\calC=\id.
\]
Note that the operators $\calA,\,\calC$ are self-adjoint, while the operator
$\calB$ is skew-symmetric. It is interesting to note that the negative index
of $\calL$ is discussed in \citep[Lemma~1]{kostenko:otd02}, where it is
stated that
\[
\rmn(\calL)=\rmn(\calH''(\phi)+\calT_\omega^2).
\]
The number of negative directions of $\calL$ is precisely the number of
negative directions associated with the linearization of \eref{e:sh4} about
$u=U$.

With respect to the spectrum of the pencil \eref{e:sh6} the result of
\autoref{thm:index} says the following. The assumptions associated with the
symmetries present in the problem imply that
\[
\ker(\calH''(U)+\calT_\omega^2)=\Span\{\calT_\omega U\},
\]
so that $\rmz(\calH''(U)+\calT_\omega^2)=n$. Furthermore, these assumptions
imply that
\[
\calT_\omega:\ker(\calH''(U)+\calT_\omega^2)\mapsto\ker(\calH''(U)+\calT_\omega^2)^\perp,
\]
so that the generalized kernel for the pencil has (at least) dimension $2n$.
Since $\calC=\id$, under the assumption that the matrix
\[
(\id-4\calT_\omega(\calH''(U)+\calT_\omega^2)^{-1}\calT_\omega)|_{\Span\{\calT_\omega U\}}
\]
is invertible, it will then be the case that the instability index count
satisfies
\begin{equation}\label{e:sh7}
k_\rmr+k_\rmc+k_\rmi^-=\rmn(\calH''(\phi)+\calT_\omega^2)-
\rmn\left((\id-4\calT_\omega(\calH''(U)+\calT_\omega^2)^{-1}\calT_\omega)|_{\Span\{\calT_\omega U\}}\right)
\end{equation}

Now that the spectral problem is understood, consider the orbital stability
of the wave. This result follows almost immediately for
\citep[Theorem~4.1]{grillakis:sto90}. An alternate interpretation of that
result is as follows. In the language of that paper the wave is said to be
orbitally stable if the reduced Hamiltonian, which is the Hamiltonian
restricted to the closed subspace orthogonal to the generalized kernel of
$\calJ\calL$, is positive definite. As was discussed in, e.g.,
\citep{deconinck:ots10,kapitula:ots07,deconinck:tos10}, this condition is
equivalent to saying that for the linearized problem \eref{e:sh5} the
spectrum is purely imaginary and satisfies $k_\rmi^-=0$. 
Under this spectral assumption, and the compactness assumptions associated
with the operators, the wave is then a local minimizer for the Lagrangian,
and hence is orbitally stable.

\begin{theorem}\label{thm:orbitalstable}
Suppose that for the quadratic pencil
\[
\calP_2(\lambda)\coloneqq(\calH''(U)+\calT_\omega^2)+\lambda(2\calT_\omega)+\lambda^2\id,
\]
which is the spectral problem for the linearization of the second-order
Hamiltonian system \eref{e:sh4a} about the steady-state $u=U$, the operators
satisfy the assumptions associated with \autoref{thm:index}. Assume that
solutions to \eref{e:sh4a} exist globally in time. If the eigenvalues satisfy
the instability index count $k_\rmr=k_\rmi^-=k_\rmc=0$ (see \eref{e:sh7}),
then the wave is orbitally stable. In other words, for each $\epsilon>0$
there is a $\delta>0$ such that if
\[
\|u(0)-U\|_X+\|\partial_t u(0)-\calT_\omega U\|_{X_1}<\delta,
\]
then
\[
\sup_{t>0}\inf_{g\in G}\left(\|u(t)-\calT(g)U\|_X+\|\partial_tu(t)-\calT(g)\calT_\omega U\|_{X_1}\right)<\epsilon.
\]
\end{theorem}

\subsection{Example: periodic waves to the ``good" Boussinesq equation}\label{s:gbou}

The generalized ``good" Boussinesq equation (gB) is of the form
\begin{equation}\label{e:a51}
\partial_t^2u+\partial_x^2(\partial_x^2u-u+f(u))=0,
\end{equation}
where $f:\R\mapsto\R$ is smooth. In traveling coordinates, i.e., $\xi=x-ct$ with $c\in(-1,1)$,
the gB can be rewritten as
\begin{equation}\label{e:a52}
\partial_t^2u-2c\partial_{t\xi}^2u+\partial_\xi^2(\partial_\xi^2u-(1-c^2)u+f(u))=0.
\end{equation}
The interest will be on solutions to \eref{e:a52} which are $2L$-periodic in
$\xi$, i.e., $u(\xi+2L,t)=u(\xi,t)$.

In order to study the existence, spectral, and orbital stability problems, it
is convenient to recast the gB \eref{e:a51} in a Hamiltonian formulation
similar to that of \eref{e:sh2}. Herein this task will be accomplished via a
trick presented in \citep{bona:geo88}. The evolution is considered to take
place on the space $L^2_\mathrm{per}[-L,+L]$, i.e., the space of
square-integrable functions which are $2L$-periodic in $\xi$. The
inner-product is the standard one, i.e.,
\[
\langle f,g\rangle=\int_{-L}^{+L}f(x)\overline{g(x)}\,\rmd x.
\]
It is straightforward to check that the original gB \eref{e:a51} is
equivalent to the system
\begin{equation}\label{e:54}
\partial_t\vu=\calJ\widehat{\calH}'(\vu),\quad\vu=(u,v)^\rmT,
\end{equation}
where $\partial_xv=\partial_tu$,
\[
\calJ=\left(\begin{array}{cc}0&\partial_x\\\partial_x&0\end{array}\right),\quad
\widehat{\calH}(\vu)=\int_{-L}^{+L}\left[\frac12(\partial_xu)^2+\frac12u^2-F(u)
+\frac12v^2\right]\,\rmd x.
\]
Here $F'(u)=f(u)$. Note that the above formulation of $\widehat{\calH}$ is
consistent with the formulation of the previous section, i.e.,
\[
\widehat{\calH}(\vu)=\calH(u)+\frac12\langle v,v\rangle,\quad
\calH(u)=\int_{-L}^{+L}\left[\frac12(\partial_xu)^2+\frac12u^2-F(u)\right]\,\rmd x,
\]
while the skew-symmetric operator $\calJ$ no longer has the property of
having a bounded inverse. The system is invariant under spatial translation,
i.e., $\widehat{\calT}(\omega)\vu(x,t)=\vu(x+\omega,t)$. Consequently, upon
using the fact that on $\ker(\partial_x)^\perp$ it is true that
\[
\calJ^{-1}\widehat{\calT}_\omega=\left(\begin{array}{cc}0&1\\1&0\end{array}\right),
\]
the conserved quantity associated with the spatial translation is given by
\[
\calQ(\vu)=\langle u,v\rangle\qquad\left(=\frac12\partial_t\langle u,u\rangle\right),
\]
and the Lagrangian for the system is
\[
\Lambda(\vu)=\widehat{\calH}(\vu)+c\calQ(\vU)\quad\Rightarrow\quad
\Lambda'(\vu)=\left(\begin{array}{c}\calH'(u)+cv\\cu+v\end{array}\right).
\]
In conclusion, the system to be studied is
\begin{equation}\label{e:55}
\partial_t\vu=\calJ\Lambda'(\vu),
\end{equation}
which is equivalent to $\partial_t\vu=\widehat{\calH}'(\vu)$ in traveling
coordinates $\xi=x-ct$.

First consider the existence problem. Since $\ker(\partial_x)=\Span\{1\}$,
for real-valued parameters $a,b$ the problem is
\begin{equation}\label{e:56}
-\partial_x^2u+u-f(u)+cv=-a,\quad v+cu=b
\end{equation}
which is equivalent to
\begin{equation}\label{e:57a}
-\partial_x^2u+(1-c^2)u-f(u)=-(a+cb).
\end{equation}
This is a well-studied problem. In order to use the desired geometric
formulation of \citet{bronski:ait11}, it will first be necessary to rescale
the wave-speed via
\begin{equation}\label{e:defch}
\hat{c}\coloneqq1-c^2\quad\Rightarrow\quad
c=c_\pm\coloneqq\pm\sqrt{1-\hat{c}},
\end{equation}
so that \eref{e:57a} can be rewritten as
\begin{equation}\label{e:57}
-\partial_x^2u+\hat{c}u-f(u)=-(a+cb).
\end{equation}
Note that in \eref{e:57} the wave-speed $c$ is either of $c_\pm$. Without
loss of generality assume that $b=0$. A periodic steady-state, say $U$, will
be a solution to the ODE
\begin{equation}\label{e:53}
\partial_\xi^2U-\hat{c}U+f(U)=a;
\end{equation}
hence, the solution will naturally depend on the parameters $a$ and
$\hat{c}$. If one sets
\[
E=\frac12(\partial_\xi U)^2+V(U,a,\hat{c}),\quad
V(U,a,\hat{c})\coloneqq-aU-\frac12\hat{c}U^2+F(U),
\]
then under the assumption that $E,\,a,$ and $\hat{c}$ are chosen so that
\begin{enumerate}
\item $E=V(U,a,\hat{c})$ has (at least) two real roots $U_\pm$ with $U_-<U_+$
\item $V(U,a,\hat{c})<E$ for $U_-<U<U_+$
\end{enumerate}
(see \autoref{f:ExistenceCriteria}) there will be a periodic solution with
period $2L$, where
\[
L=\frac1{\sqrt{2}}\int_{U_-}^{U_+}\frac{\rmd U}{\sqrt{E-V(U,a,\hat{c})}}.
\]
As it will be seen, in particular examples the spectral stability of the
$2L$-periodic solution will naturally depend upon the parameters
$a,\hat{c},E$. The dependence of the solution on these parameters will be
implicit in all that follows.

\begin{figure}
\centering
\includegraphics{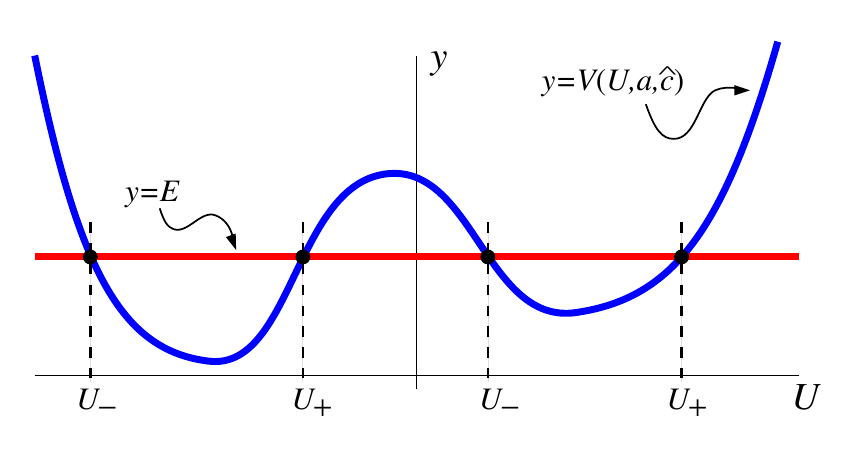}
\caption{(color online) The criteria that the potential $V(U,a,\hat{c})$ must satisfy relative
to the energy $E$ in order for there to exist spatially periodic solutions. Here the
energy was chosen so that there exist two distinct periodic solutions. Increasing the energy
past a threshold $E^*$, for which there are two homoclinic orbits, yields that the two
solutions merge into one periodic solution.}
\label{f:ExistenceCriteria}
\end{figure}

Now consider the stability problem. Let the $2L$-periodic wave found in
\eref{e:57} be denoted by $u=U$. Recalling that we set $b=0$, from
\eref{e:56} the $v$-component of the wave is given by $v=-cU$, where from
\eref{e:defch} one has $c=c_\pm$ can be of either sign for a fixed value of
$\hat{c}$. The steady-state solution for \eref{e:55} is then given by
\[
\vu=\vU=\left(\begin{array}{c}U\\-cU\end{array}\right).
\]
Under the mapping
\begin{equation}\label{e:58}
\vu=\vU+\vv
\end{equation}
the system \eref{e:55} becomes
\[
\partial_t\vv=\calJ\Lambda'(\vU+\vv).
\]
The evolution problem must now be considered on a space for which $\calJ$ has
bounded inverse, i.e., on the space of mean-zero functions. Let
$\Pi_0:L^2_\mathrm{per}[-L,+L]\mapsto H_0$ be the self-adjoint projection
operator
\[
\Pi_0u=u-\frac1{2L}\langle u,1\rangle;
\]
in other words, $\Pi_0$ is the orthogonal projection onto
$\ker(\partial_\xi)^\perp$. Here
\[
H_0\coloneqq\{u\in L^2_{\mathrm{per}}[-L,+L]:\langle u,1\rangle=0\}=
\ker(\partial_\xi)^\perp.
\]
When writing $\Pi_0\vv$ it will be implicitly assumed that $\Pi_0$ is being
applied to each component of $\vv$. In \citep[Section~2]{deconinck:ots10} it
is shown that the proper evolution equation to consider is
\begin{equation}\label{e:59}
\partial_t\vv=\calJ\Pi_0\Lambda'(\vU+\vv),\quad\vv(0)=\vv_0,
\end{equation}
where $\Pi_0\vv_0=\vv_0$ implies that $\Pi_0\vv(t)=\vv(t)$ for all $t>0$. In
other words, \eref{e:59} describes the evolution of mean-zero perturbations
of the underlying wave. Since the evolution occurs on $H_0\times H_0$, and
$\partial_x:H_0\mapsto H_0$ has bounded inverse, in this formulation the
operator $\calJ$ now has bounded inverse.

First consider the spectral stability problem. The linearized eigenvalue
problem
\[
\lambda\vv=\calJ\Pi_0\Lambda''(\vU)\vv,\quad\Pi_0\vv=\vv
\]
can be rewritten as
\[
\partial_x[\Pi_0\calH''(U)\Pi_0u+cv]=\lambda v,\quad\partial_x[cu+v]=\lambda u,
\]
where $u,v\in H_0$. Differentiating the first equation yields
\[
\partial_x^2\Pi_0\calH''(U)\Pi_0 u+c\partial_x(\partial_xv)=\lambda\partial_xv,
\]
and substituting the second equation into the first and simplifying gives the
quadratic pencil problem
\[
\left[\lambda^2-2c\lambda\partial_x+\partial_x^2(\Pi_0(-\calH''(U)+c^2)\Pi_0)\right]u=0.
\]
Since $u\in H_0$, one can write $v=\partial_x^{-1}u\in H_0$, so that the
pencil becomes
\[
\partial_x\left[\lambda^2-2c\lambda\partial_x-\partial_x(\Pi_0(\calH''(U)-c^2)\Pi_0)\partial_x\right]v=0.
\]
Since $\partial_x$ has bounded inverse, upon setting $\calL_2$ to be the
well-understood self-adjoint Hill operator
\[
\calL_2=-\partial_x^2+\hat{c}-f'(U(x)),
\]
the pencil problem to be studied is
\[
\left[\lambda^2-2c\lambda\partial_x-\partial_x(\Pi_0\calL_2\Pi_0)\partial_x\right]v=0,\quad
v\in H_0.
\]
Note that in the notation of the previous section,
\begin{equation}\label{e:59a}
\calC=\id,\quad\calB=-2c\partial_x,\quad\calA=-\partial_x(\Pi_0\calL_2\Pi_0)\partial_x.
\end{equation}

Before proceeding with the spectral analysis, the assumptions on the
operators given in \autoref{thm:index} must be verified. First consider
$\ker(\calA)$. Since $\calL_2(\partial_\xi U)=0$, it is true that
\[
\calL_2\Pi_0\cdot\partial_\xi(U-\overline{U})=0,\quad
\overline{U}=\frac1{2L}\int_{-L}^{+L}U(x)\,\rmd x.
\]
In other words, $U-\overline{U}\in\ker(\calA)$. In order for $\ker(\calA)$ to
have another linearly independent element, it must be the case that
$\calL_2^{-1}(1)\in H_0$. It will be henceforth assumed that no other element
in the kernel exists, i.e.,
\begin{equation}\label{e:59b}
\langle\calL_2^{-1}(1),1\rangle\neq0,
\end{equation}
so that
\[
\ker(\calA)=\Span\{U-\overline{U}\}.
\]
Letting
\[
P_\calA:H_0\mapsto\Span\{U-\overline{U}\},\quad%
P_\calA^\perp:H_0\mapsto\Span\{U-\overline{U}\}^\perp\subset H_0
\]
be orthogonal projections, it must be checked that
\[
(P_\calA^\perp\calA P_\calA^\perp)^{-1}P_\calA^\perp\calB P_\calA^\perp,\,%
(P_\calA^\perp\calA P_\calA^\perp)^{-1}P_\calA^\perp\calC P_\calA^\perp
\]
are compact operators. This immediately follows from the fact that
$(P_\calA^\perp\calA P_\calA^\perp)^{-1}$ is compact, and both $\calB$ and
$\calC$ are differentiable operators of lesser order than $\calA$.  Together, the above considerations
verify the hypothesis of \autoref{l:noessential}.

Next, we must verify that the operators $\calA$, $\calB$, and $\calC$ are $\calS$-compact for some
compact operator $\calS$.  For each $\alpha>0$, define the operator
\[
\calS_\alpha\coloneqq\Pi_0(\partial_x^2+1)^\alpha\Pi_0
\]
acting on $L^2_{\rm per}([-L,+L])$.
It is clear that $\calS_\alpha^{-1}$ is a compact self-adjoint operator on $H_0$ for each $\alpha>0$;
furthermore, it is true that the space
\[
X_{\calS_\alpha}=\{u\in H_0:\langle\Pi_0(\partial_x^2+1)^{2\alpha}\Pi_0u,u\rangle<\infty\}
\]
is dense for any $\alpha>0$.  Now, clearly the operator
$\calS_\alpha^{-1}\calC\calS_\alpha^{-1}=\calS_\alpha^{-2}$ is compact for any $\alpha>0$.
Regarding the operator $\calB$, it is easy to see that
$\calS_\alpha^{-1}\partial_x\calS_\alpha^{-1}$ will be compact as long as $1/4<\alpha$.
Finally, the operator $\calS_\alpha^{-1}\calA\calS_\alpha^{-1}$ will have a compact
resolvent as long as $0<\alpha<1$. In conclusion, as long as $1/4<\alpha<1$,
the operators will be $\calS_\alpha$-compact, thus verifying hypothesis (a) and (b) of \autoref{thm:index}.

From the skew-symmetry of the operator and the fact that $U$ is $2L$-periodic
it is clear that
\[
\calB|_{\ker(\calA)}=-2c\langle U-\overline{U},\partial_x(U-\overline{U})\rangle=0.
\]
Provided that $(\id-4c^2\partial_x\calA^{-1}\partial_x)|_{\ker(\calA)}$ is invertible then,
a direct application of \autoref{thm:index}, using the explicit form of the operators
given in \eref{e:59a} and noting that $\calC=\id$ is clearly a positive definite
operator, implies that the index count satisfies
\[
k_\rmr+k_\rmc+k_\rmi^-=\rmn(\calA)-%
\rmn((\id-4c^2\partial_x\calA^{-1}\partial_x)|_{\ker(\calA)}).
\]
In order to compute $\rmn(\calA)$, first note that
\[
\langle u,\calA u\rangle=\langle u,-\partial_x(\Pi_0\calL_2\Pi_0)\partial_x u\rangle=
\langle\partial_x u,\Pi_0\calL_2\Pi_0(\partial_x u)\rangle.
\]
Thus, upon using the fact that $\partial_x:H_0\mapsto H_0$ has a bounded
inverse it is clear that
\[
\rmn(\calA)=\rmn(\Pi_0\calL_2\Pi_0).
\]
Regarding the quantity on the right, it was shown in
\citep[equation~(2.25)]{deconinck:ots10} that if the inequality of
\eref{e:59b} holds, then
\[
\rmn(\Pi_0\calL_2\Pi_0)=\rmn(\calL_2)-\rmn(\langle\calL_2^{-1}(1),1\rangle).
\]
Consequently, it can now be said that
\begin{equation}\label{e:510}
\rmn(\calA)=\rmn(\calL_2)-\rmn(\langle\calL_2^{-1}(1),1\rangle),
\end{equation}
so that the index count satisfies
\[
k_\rmr+k_\rmc+k_\rmi^-=\rmn(\calL_2)-\rmn(\langle\calL_2^{-1}(1),1\rangle)-%
\rmn((\id-4c^2\partial_x\calA^{-1}\partial_x)|_{\ker(\calA)}).
\]

Recalling that $c^2=1-\hat{c}$, the index count is complete once the scalar
\[
(\id-4(1-\hat{c})\partial_x\calA^{-1}\partial_x)|_{\ker(\calA)}
\]
is computed. From \eref{e:59a} one has that
\[
\begin{split}
\id-4(1-\hat{c})\partial_x\calA^{-1}\partial_x&=\id+
4(1-\hat{c})\partial_x\cdot
\partial_x^{-1}(\Pi_0\calL_2\Pi_0)^{-1}\partial_x^{-1}
\cdot\partial_x\\
&=\id+4(1-\hat{c})(\Pi_0\calL_2\Pi_0)^{-1}:
\end{split}
\]
the second line follows from the fact that $\partial_x$ is invertible on
$H_0$. Using the characterization of $\ker(\calA)$, it is then seen that
\[
(\id-4(1-\hat{c})\partial_x\calA^{-1}\partial_x)|_{\ker(\calA)}=
\langle U-\overline{U},U-\overline{U}\rangle+4(1-\hat{c})
\langle(\Pi_0\calL_2\Pi_0)^{-1}(U-\overline{U}),U-\overline{U}\rangle.
\]
Finally, in the study of the orbital stability of periodic waves for the
generalized Korteweg-de Vries equation (gKdV) it was shown in
\citep[Section~3]{deconinck:ots10} that
\[
\langle(\Pi_0\calL_2\Pi_0)^{-1}(U-\overline{U}),U-\overline{U}\rangle=
D_{\mathrm{gKdV}},\quad
D_{\mathrm{gKdV}}\coloneqq\frac{\left|%
\begin{array}{cc}%
\langle\calL_2^{-1}(U),U\rangle&\langle\calL_2^{-1}(U),1\rangle\\
\langle\calL_2^{-1}(U),1\rangle&\langle\calL_2^{-1}(1),1\rangle
\end{array}%
\right|}%
{\langle\calL_2^{-1}(1),1\rangle}.
\]

\begin{lemma}\label{lem:51}
Consider the quadratic pencil \eref{e:59}. If
$\langle\calL_2^{-1}(1),1\rangle\neq0$, then the stability index is given by
\[
k_\rmr+k_\rmi^-+k_\rmc=\rmn(\calL_2)-\rmn(\langle\calL_2^{-1}(1),1\rangle)-
\rmn(\langle U-\overline{U},U-\overline{U}\rangle+4(1-\hat{c})D_{\mathrm{gKdV}}),
\]
where
\[
D_{\mathrm{gKdV}}\coloneqq\frac{\left|%
\begin{array}{cc}%
\langle\calL_2^{-1}(U),U\rangle&\langle\calL_2^{-1}(U),1\rangle\\
\langle\calL_2^{-1}(U),1\rangle&\langle\calL_2^{-1}(1),1\rangle
\end{array}%
\right|}%
{\langle\calL_2^{-1}(1),1\rangle},
\]
and the parameter $\hat{c}$ is related to the original wave-speed $c$ via
$c^2=1-\hat{c}$.
\end{lemma}

\begin{remark}
It was shown in \citep[Theorem~2.6]{deconinck:ots10} that the stability index
for the gKdV is given by
\[
k_\rmr+k_\rmi^-+k_\rmc=\rmn(\calL_2)-\rmn(\langle\calL_2^{-1}(1),1\rangle)-
\rmn(D_{\mathrm{gKdV}});
\]
hence, when studying the spectrum for periodic waves to gB and gKdV there is
an intimate connection in the indices for the two problems. If
$D_{\mathrm{gKdV}}<0$, then for
\[
1>\hat{c}>1+
\frac{\langle U-\overline{U},U-\overline{U}\rangle}{4D_{\mathrm{gKdV}}}
\]
the stability index for the gB is exactly the same as for the gKdV.
Otherwise, there is precisely one more eigenvalue which is counted by the
index. On the other hand, if $D_{\mathrm{gKdV}}>0$, then the index for the
quadratic pencil is exactly that for the gKdV equation.
\end{remark}

\begin{remark}
There is a geometric interpretation associated with the quantity
$D_{\mathrm{gKdV}}$.  The interested reader should consult
\citep{bronski:ait11} for more details.
\end{remark}

\begin{remark}\label{hakkeavrem}
In \citet{hakkaev:lsa12} the spectral stability problem was considered under
the additional assumption that $\rmn(\calL_2)=1$. Furthermore, while it is
not explicitly stated, they further assume that
$\langle\calL_2^{-1}(1),1\rangle>0$, so that (in this paper's notation) the
index becomes
\[
k_\rmr+k_\rmi^-+k_\rmc=1-
\rmn(\langle U-\overline{U},U-\overline{U}\rangle+4(1-\hat{c})D_{\mathrm{gKdV}}).
\]
The instability criterion in that paper follows from the fact that
$k_\rmi^-,k_\rmc$ must be even integers, so that $k_\rmi^-=k_\rmc=0$, with
\[
k_\rmr=\begin{cases}
1,\quad&\hat{c}>\hat{c}^*\\
0,\quad&\hat{c}<\hat{c}^*,
\end{cases}
\]
where
\[
\hat{c}^*=1+
\frac{\langle U-\overline{U},U-\overline{U}\rangle}{4D_{\mathrm{gKdV}}}.
\]
It is not clear if in that paper an explicit connection is shown between the
gB and the gKdV.
\end{remark}

Now consider the orbital stability problem. The local global well-posedness
problem has been studied in, e.g.,
\citep{farah:otp10,fang:eau96,arruda:nsp09}, and it will henceforth be
assumed that the problem can be solved (at least) locally. Depending on the
growth rate of the nonlinearity, this implies that the initial data for
\eref{e:59} satisfies, e.g., $v_1(0)\in H^1_{\mathrm{per}}[-L,+L]$, and
$v_2(0)\in H^{-1}_{\mathrm{per}}[-L,+L]$, where the norm for the latter space
is given by
\[
\|u\|_{H^{-1}}^2=\sum_{z\in\Z}\frac{|\hat{u}(z)|^2}{1+|z|^2}.
\]
The form of the Lagrangian in \eref{e:55} makes clear that in order to
control the nonlinear terms the proper space in which to work is
$H^1_{\mathrm{per}}[-L,+L]\times L^2_{\mathrm{per}}[-L,+L]$. It will be
further assumed that the hypothesis leading to \autoref{lem:51} hold, and
that the spectral problem has zero instability index, i.e.,
$k_\rmr=k_\rmc=k_\rmi^-=0$. As was seen in
\citep[Section~2.4]{deconinck:ots10}, this is sufficient in order to conclude
that the wave is orbitally stable with respect to the evolution defined by
\eref{e:59}.

\begin{proposition}\label{thm:orbstable}
Suppose that the IVP for \eref{e:59} is locally well-posed. Further suppose
that in addition to what is required for a unique local solution to exist,
the mean-free perturbative initial data for the system \eref{e:59} satisfies
$\vv(0)\in H^1_{\mathrm{per}}[-L,+L]\times L^2_{\mathrm{per}}[-L,+L]$. If the
spectral problem satisfies $k_\rmr=k_\rmc=k_\rmi^-=0$, then the underlying
wave is orbitally stable. In other words, for each $\epsilon>0$ there is a
$\delta>0$ such that for \eref{e:59},
\[
\|\vu(0)-\vU\|_{H^1_{\mathrm{per}}\times L^2_{\mathrm{per}}}<\delta\quad\Rightarrow\quad
\sup_{t>0}\inf_{\omega\in\R}\|\vu(t)-\widehat{\calT}(\omega)\vU\|_{H^1_{\mathrm{per}}\times L^2_{\mathrm{per}}}<\epsilon.
\]
\end{proposition}

Note that for the original system \eref{e:a51} the requirement on the initial
data is
\[
u(0)=U+v_1(0),\quad\partial_tu(0)=-c\partial_xU+\partial_xv_2(0),
\]
where each component $v_j(0)$ has zero mean. A natural question is then: what
happens if the initial perturbation is not mean-free? In this case, we now
argue for orbital stability with respect to a nearby periodic traveling wave
of \eref{e:59}.  For related arguments in the contexts of other nonlinear
dispersive equations, see the work of \citet{HG07} on the nonlinear
Schrodinger equation as well as the works of \citet{Jkdv,Jbbm} on generalized
KdV and BBM models, respectively.

To begin, we make a few comments regarding the conserved quantities of
\eref{e:54}.   As discussed above, even though we set $b=0$ in the analysis,
the periodic traveling wave solutions of \eref{e:59}, which are solutions to
the ODE \eref{e:57a}, form a five-parameter family of solutions of the form
\[
u_{\xi}(x,t)=u(x-ct+\xi;a,E,c,b)
\]
where $\xi\in\mathbb{R}$ and $(a,E,c,b)$ belong to some open set
$\Omega\in\mathbb{R}^4$.  Furthermore, the evolution equation \eref{e:54}
admits the following two conserved quantities: the momentum (charge)
\[
\mathcal{P}(\vw)\coloneqq\int_{-L}^Lw_1w_2\,\rmd x,\quad
\vw\coloneqq\left(w_1,w_2\right)\in H^1_{\rm per}[-L,+L]\times L^2_{\mathrm{per}}[-L,+L],
\]
arising from the translation invariance of \eref{e:59}, and the casimirs
\[
\mathcal{M}_1(\vw)\coloneqq\int_{-L}^Lw_1\,\rmd x,\quad
\mathcal{M}_2(\vw\coloneqq\int_{-L}^Lw_2\,\rmd x,\quad
\vw\coloneqq\left(w_1,w_2\right)\in H^1_{\rm per}[-L,+L]\times L^2_{\mathrm{per}}[-L,+L],
\]
arising from the fact that $\ker(\mathcal{J})$ is non-trivial.  Notice that
$\mathcal{P}$ and $\mathcal{M}$ are smooth functionals on $H^1_{\rm
per}[-L,+L]\times L^2_{\mathrm{per}}[-L,+L]$ and that, when restricted to the
manifold of traveling wave solutions of \eref{e:59} the functionals
$\mathcal{M}_1,\,\mathcal{M}_2$, and $\mathcal{P}$ reduce to
\[
\quad M_1(a,E,c,b)\coloneqq\int_0^Tu(x;a,E,c,b)\,\rmd x,\quad M_2(a,E,c,b):=cM_1(a,E,c,b)-bT,
\]
and
\[
\widetilde P(a,E,c,b)\coloneqq-P(a,E,c,b)+bM_1(a,E,c,b),
\]
where here $T=T(a,E,c,b)\,(=2L)$ denotes the period of the wave and
$P(a,E,c,b)\coloneqq c\int_0^Tu(x;a,E,c,b)^2\,\rmd x$.

Now, consider the case where the means of $v_j(0)$ are small, but non-zero.
Using the geometric formalism of Bronski et.al (see
\citep{bronski:ait11,brj}) we have the following key lemma.

\begin{lemma}\label{l:geom}
With the notation as above, we the equality
\[
\left(\mathcal{I}-4(1-\hat{c})\partial_x\mathcal{A}^{-1}\partial_x\right)|_{{\rm ker}(\mathcal{A})}
=T{\rm det}\left(\begin{array}{cc}
                 T_a & M_{1,a}\\
                 T_E & M_{1,E}
                 \end{array}\right)
{\rm det}\left(\begin{array}{cccc}
                 T_E & M_{1,E} & \widetilde P_E & M_{2,E}\\
                 T_a & M_{1,a} & \widetilde P_a & M_{2,a}\\
                 T_c & M_{1,c} & \widetilde P_c & M_{2,c}\\
                 T_b & M_{1,b} & \widetilde P_b & M_{2,b}
                 \end{array}
                 \right).
\]
\end{lemma}

\begin{proof}
To compute the left hand side, define the function
\[
\phi_2(x):=\det\left(\begin{array}{ccc}
                     u_a & T_a & M_{1,a}\\
                     u_E & T_E & M_{1,E}\\
                     u_c & T_c & M_{1,c}
                     \end{array}\right)
\]
and satisfies
\[
\Pi_0\mathcal{L}_2\Pi_0\phi_2 = \Pi_0\mathcal{L}_2\phi_2 =
2c{\rm det}\left(\begin{array}{cc}
                 T_a & M_{1,a}\\
                 T_E & M_{1,E}
                 \end{array}\right)\left(U-\bar U\right).
\]
The stated equality now follows by directly calculating
\[
\left(\mathcal{I}-4(1-\hat{c})\partial_x\mathcal{A}^{-1}\partial_x\right)|_{{\rm ker}(\mathcal{A})}
=\left<\Pi_0 U,\left(1+4c^2\left(\Pi_0\mathcal{L}_2\Pi_0\right)^{-1}\right)\Pi_0 U\right>
\]
and comparing to the right hand side of the above equality.
\end{proof}

%
%
%

From \autoref{l:geom} and the assumption that
\[
\left(\mathcal{I}-4(1-\hat{c})\partial_x\mathcal{A}^{-1}\partial_x\right)|_{{\rm ker}(\mathcal{A})}
\]
is nonsingular at the underlying wave $U$, corresponding say to $(a,E,c,b)=(a_0,E_0,c_0,b_0)$,
implies that the map
\[
\mathbb{R}^4\ni(a,E,c,b)\mapsto\left(T(a,E,c,b),M_1(a,E,c,b),\widetilde P(a,E,c,b),M_2(a,E,c,b)\right)\in\mathbb{R}^4
\]
is a local diffeomorphism from a neighborhood of $(a_0,E_0,c_0,b_0)$ onto a neighborhood of
the point
\[
(T,M_1,\widetilde P,M_2)(a_0,E_0,c_0,b_0).
\]
It follows that we can find a curve $[0,1]\ni s\mapsto
(a(s),E(s),c(s),b(s))\in\mathbb{R}^4$ with
$(a(0),E(0),c(0),b(0))=(0,0,0,0)$
such that for
each $s\in[0,1]$ the function
\[
\tilde{u}(x;s)=u\left(x;a_0+a(s),E_0+E(s),c_0+c(s),b_0+b(s)\right)
\]
is a $T=T(a_0,E_0,c_0,b_0)$-periodic traveling wave solution of \eref{e:a51}
and that, moreover the endpoint condition
\[
\begin{split}
M_j\left(a_0+a(1),E_0+E(1),c_0+c(1),b_0+b(1)\right)&=\mathcal{M}_j\left(\vu(0)+\vv(0)\right),\quad j=1,2\\
P\left(a_0+a(1),E_0+E(1),c_0+c(1),b_0+b(1)\right)&=\mathcal{P}\left(\vu(0)+\vv(0)\right)
\end{split}
\]
and growth constraint
\[
\sup_{s\in(0,1)}\left|\left(a(s),E(s),c(s),b(s)\right)\right|_{\mathbb{R}^4}\lesssim \left\|\vv(0)\right\|_{H^1_{\rm per}(-L,L)\times L^2_{\rm per}(-L,L)}
\]
are satisfied.
Assuming $\left\|\vv(0)\right\|_{H^1_{\rm per}\times L^2_{\mathrm{per}}}$ is
sufficiently small, it follows that the wave $\tilde{u}(\cdot,1)$ is
nonlinearly orbitally stable in the sense described in
\autoref{thm:orbstable}, which, by the triangle inequality, implies orbital
stability of $U$ to initial perturbations $\vv(0)$ with nonzero, but
sufficiently small, mean.  Since $\mathcal{M}_j$ and $\mathcal{P}$ are
continuous in the $H^1_{\rm per}\times L^2_{\mathrm{per}}$ topology, it
follows that we have orbital stability in the standard sense without the
restriction to mean-free initial data in \eref{e:59}. This observation yields
the following extension of \autoref{thm:orbstable}

\begin{theorem}\label{thm:orbstable2}
Suppose that the IVP for \eref{e:59} is locally well-posed.  Further, suppose
that in addition to what is required for a unique local solution to exist,
the perturbative initial data for the system \eref{e:59} satisfies $\vv(0)\in
H^1_{\rm per}[-L,+L]\times L^2_{\mathrm{per}}[-L,+L]$.  If the spectral
problem satisfies $k_r=k_c=k_i^-=0$, then the underlying wave is orbitally
stable, i.e., for each $\epsilon>0$ there exists a $\delta>0$ such that for
\eref{e:59} we have
\[
\|\vu(0)-\vU\|_{H^1_{\mathrm{per}}\times L^2_{\mathrm{per}}}<\delta\quad\Rightarrow\quad
\sup_{t>0}\inf_{\omega\in\R}\|\vu(t)-\widehat{\calT}(\omega)\vU\|_{H^1_{\mathrm{per}}\times L^2_{\mathrm{per}}}<\epsilon.
\]
\end{theorem}


In the next sections, we utilize the above explicit connection of the
stability problems for gB and gKdV type equations to make several comments
regarding the stability of periodic waves in the gB equation.

\subsubsection{Case study: power law nonlinearity}\label{s:case1}

The stability indices for the gKdV equation were computed in
\citet{bronski:ait11} for the case that $f(u)=u^{p+1}$ for some $p\ge1$ and
$|\hat{c}|<1$. Note via \eref{e:defch} that this implies $|c_\pm|<1$. While
we will not do it here, the gB can be rescaled so that the influence of
$\hat{c}$ is removed from the steady-state problem. This independence is
reflected in the existence diagrams. As in the analysis leading to the
statement of \autoref{lem:51}, it will be assumed here that $b=0$.

First suppose that $p=1$, which corresponds to the classical gB equation. In
\citep[Section~5.1]{bronski:ait11} it is shown that all periodic waves in this model
satisfy
\[
\rmn(\calL_2)=1,\quad\rmn(\langle\calL_2^{-1}(1),1\rangle=0,\quad
\rmn(D_{\mathrm{gKdV}})=1.
\]
Thus, via \autoref{lem:51} it is true that for a given $(a,E)$ there is a
critical positive wave-speed $\hat{c}^*_{a,E}<1$ such that
$k_\rmi^-=k_\rmc=0$ with
\[
k_\rmr=\begin{cases}
1,\quad&\hat{c}>\hat{c}_{a,E}\\
0,\quad&\hat{c}<\hat{c}_{a,E}.
\end{cases}
\]
Returning to the original wavespeed $c$, it follows that for any periodic
traveling wave solution of the classical gB equation there exists a range of
wavespeeds $(1-\hat{c}_{a,E}^*)^{1/2}<|c|<1$ for which the wave is
nonlinearly (orbitally) stable, while it is unstable spectrally unstable to
perturbations with the same period if $|c|<(1-\hat{c}_{a,E}^*)^{1/2}$. This
is consistent with the result of \citep[Theorem~2]{hakkaev:lsa12}, where
$\hat{c}_{a,E}$ is explicitly given when $a=0$. In that paper the case of
nonzero $a$ was not considered and only spectral instability for
$|c|\leq(1-\hat{c}_{a,E}^*)^{1/2}$ was verified.  Here, our calculations
compliment this result by verifying that waves traveling with speed greater
than $(1-\hat{c}_{a,E}^*)^{1/2}$ are by \autoref{thm:orbstable} indeed
nonlinearly stable.

\begin{figure}
\centering
\includegraphics[scale=.45]{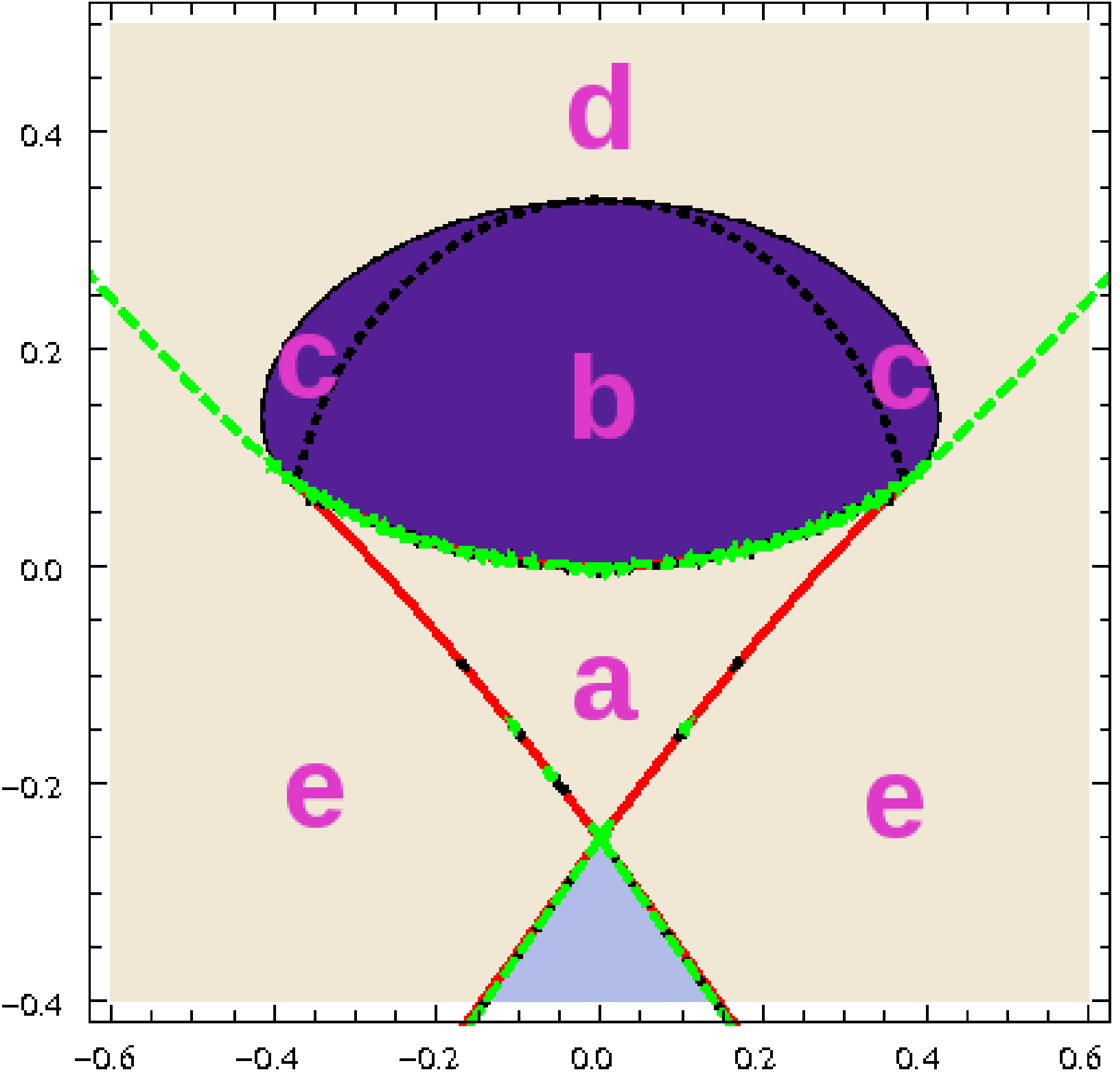}
\caption{(color online) The configuration space in the $aE$-plane when $p=2$, 
$b=0$, and $|\hat{c}|<1$ fixed
(see \citep[Figure~3]{bronski:ait11}). The swallowtail
figure divides the plane into regions containing $0,1,$ and $2$ periodic
solutions. In region (a) there are two solutions, while all of the other marked regions
have one solution. In the unmarked region there are no periodic solutions.
The quantities used in the stability calculation are given in an accompanying table.}
\label{f:J_MKdV}
\end{figure}

For examples which are not covered in \citet{hakkaev:lsa12}, e.g., when it is
possible for $\rmn(\calL_2)\ge2$, first consider the problem when $p=2$. The
table below, which corresponds to \autoref{f:J_MKdV}, can be derived from
\citep[Section~5.2]{bronski:ait11}:

\begin{center}
\begin{tabular}{|c|c|c|c|}
\hline\rule[-3mm]{0mm}{8mm}
Region &  $\rmn(\calL_2)$ &  $\rmn(\langle\calL_2^{-1}(1),1\rangle)$  & $\rmn(D_{\mathrm{gKdV}})$ \\\hline\hline
(a) & 1 & 0 & 1\\\hline
(b) & 2 & 0 & 1 \\\hline
(c) & 2 & 1 & 0\\\hline
(d) & 2 & 1 & 1\\\hline
(e) & 1 & 0 & 1\\\hline
\end{tabular}
\end{center}

\noindent From the theoretical result in \autoref{lem:51} it will be the case
that in that in regions (a), (d), and (e) there will exist a
$0<\hat{c}_{a,E}<1$ such that $k_\rmr=1$ for $\hat{c}>\hat{c}_{a,E}$ and
$k_\rmr=0$ otherwise; furthermore, it is always true that $k_\rmi^-=k_\rmc=0$
in these regions. In region (c) it will be the case that $k_\rmr=1$ for all
$c$ with the other two indices being zero. Finally, in region (b) there will
exist a $0<\hat{c}_{a,E}<1$ such that $k_\rmr=1$ for
$-1<\hat{c}<\hat{c}_{a,E}$ with the other two indices being zero, while for
$\hat{c}>\hat{c}_{a,E}$ all that can be said is that
$k_\rmr+k_\rmi^-+k_\rmc=2$. Notice, however, that by parity we see for speeds
$\hat{c}>\hat{c}_{a,E}^*$ in region (b) we have $k_\rmr=0$ and
$k_{\rmi}^-+k_\rmc=2$, which allows the possibility that some waves may still
be spectrally stable in this region with $k_{\rmi}^-=2$ or that some waves
may be spectrally unstable to perturbations with the same period with
$k_\rmc=2$: such a situation is precluded in the well-studied solitary wave
theory.

\begin{remark}
In \citep[Theorem~1]{hakkaev:lsa12} it is shown that for one of the two
solutions in region (a) of \autoref{f:J_MKdV} with $a=0$ the index satisfies
$k_\rmi^-=k_\rmc=0$ with $k_\rmr=1$ for $\hat{c}>\hat{c}_{a,E}$, and
$k_\rmr=0$ for $\hat{c}<\hat{c}_{a,E}$. Furthermore, for this solution the
constant $\hat{c}_{a,E}$ is explicitly given when $a=0$. Although they do not
consider the case in their paper, the same result holds in region (e). The
parameter region which is outside their theory comprises the union of (b),
(c), and (d).
\end{remark}

\begin{figure}
\centering
\includegraphics[scale=.45]{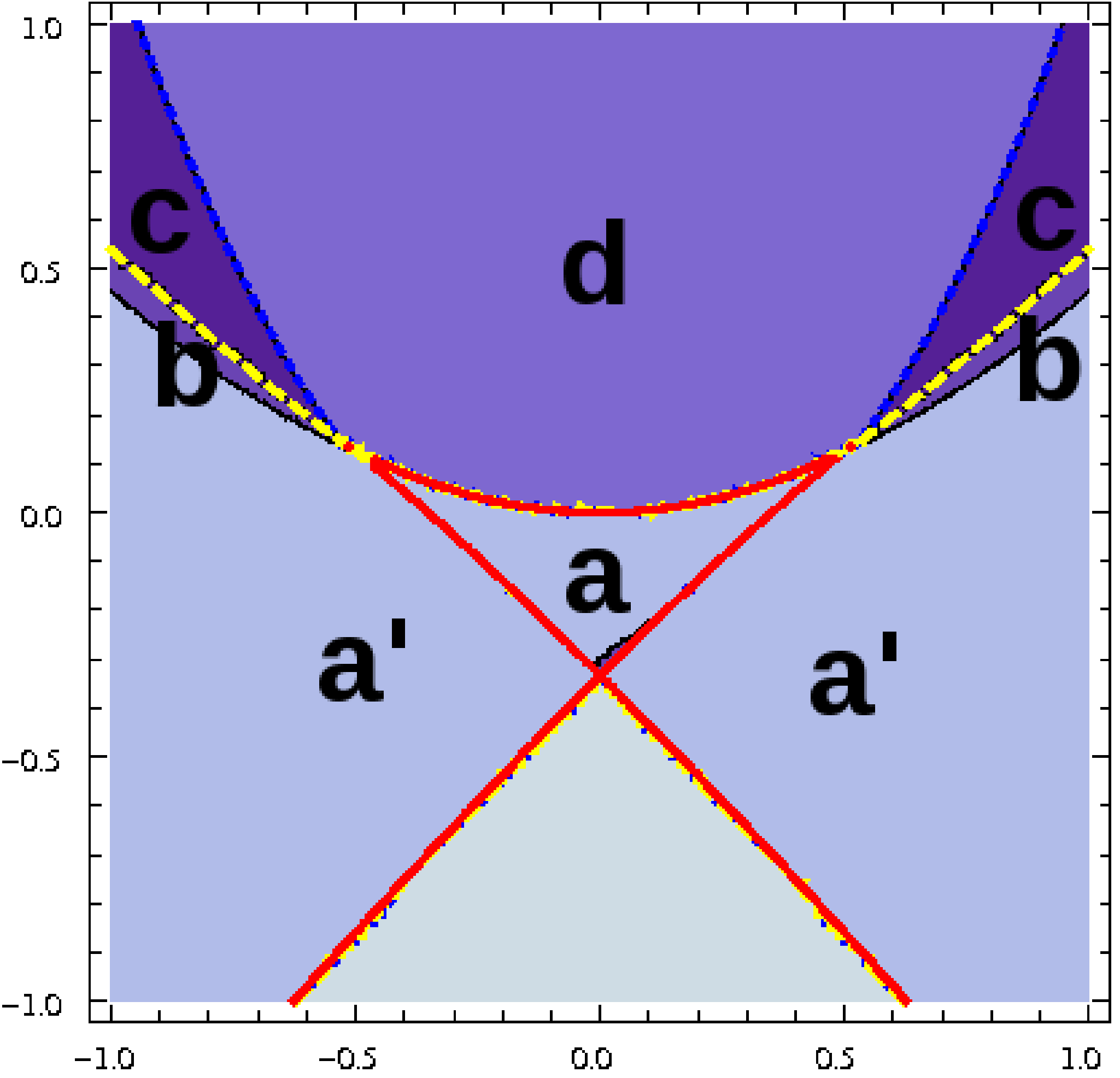}
\caption{(color online) The configuration space in the $aE$-plane when $p=4$ and $|\hat{c}|<1$ (see \citep[Figure~4]{bronski:ait11}). The swallowtail
figure divides the plane into regions containing $0,1,$ and $2$ periodic
solutions. In region (a) there are two solutions, while all of the other marked regions
have one solution. In the unmarked region there are no periodic solutions.
The quantities used in the stability calculation are given in an accompanying table.}
\label{f:KdV4_stability_diagram}
\end{figure}

The below table for $p=4$ which corresponds to
\autoref{f:KdV4_stability_diagram} can be derived from
\citep[Section~5.3]{bronski:ait11}:

\begin{center}
\begin{tabular}{|c|c|c|c|}
\hline\rule[-3mm]{0mm}{8mm}
Region &  $\rmn(\calL_2)$ &  $\rmn(\langle\calL_2^{-1}(1),1\rangle)$  & $\rmn(D_{\mathrm{gKdV}})$ \\\hline\hline
(a)  & 1 & 0 & 1\\\hline
(a') & 1 & 0 & 1 \\\hline
(b)  & 1 & 0 & 0\\\hline
(c)  & 2 & 1 & 0\\\hline
(d)  & 2 & 0 & 1\\\hline
\end{tabular}
\end{center}

\noindent From the theoretical result in \autoref{lem:51} it will be the case
that in that in regions (a) and (a') there will exist a $0<c^*_{a,E}<1$ such
that $k_\rmr=1$ for $\hat{c}>\hat{c}_{a,E}$ and $k_\rmr=0$ otherwise;
furthermore, it is always true that $k_\rmi^-=k_\rmc=0$. In regions (b) and
(c) it will be the case that $k_\rmr=1$ for all $\hat{c}$ with the other two
indices being zero. Finally, in region (d) there will exist a
$0<\hat{c}_{a,E}<1$ such that $k_\rmr=1$ for $\hat{c}<\hat{c}_{a,E}$ with the
other two indices being zero, while for $\hat{c}>\hat{c}_{a,E}$ all that can
be said is that $k_\rmr+k_\rmi^-+k_\rmc=2$.

\subsubsection{Case study, continued: solitary wave and equilibrium solution limits}\label{s:case2}

In this final section, we make some comments regarding the stability of
periodic traveling wave solutions of \eref{e:a51} which are either near the
solitary wave or near an equilibrium (constant) solution. Throughout this
section, we continue to consider \eref{e:a51} with power law nonlinearity
$f(u)=u^{p+1}$ for some $p\geq 1$.  In this case, we have from \eref{e:57a}
that the profile $U$ satisfies the ODE
\[
\partial_x^2u=(1-c^2)u-u^{p+1},
\]
where, for simplicity, we are restricting our discussion to those waves with
$a=b=0$\footnote{As we will see below, this is a natural restriction when
considering the limiting case to a solitary wave asymptotic to zero as
$x\to\pm\infty$.}.  This equation is clearly Hamiltonian, and has critical
points $(u,\partial_x u)=(0,0)$, corresponding to a saddle point, and
$(u,\partial_x u)=((1-c^2)^{1/p},0)$\footnote{Notice when $p$ is an even
integer, the point $(u,\partial_x u)=(-(1-c^2)^{1/p},0)$ is also a critical
point.  In this discussion, we ignore this additional critical point, noting
that any conclusions for periodic waves emerging from the $((1-c^2)^{1/p},0)$
critical point hold also for those emerging from the $(-(1-c^2)^{1/p},0)$
critical point.  For general $p\geq 1$, the governing ODE does not admit such
negative solutions.}, corresponding to a nonlinear center. Further, for a
fixed wavespeed $c\in(-1,1)$, in the two-dimensional $(u,\partial_x u)$
phase plane the nonlinear center $((1-c^2)^{1/p},0)$ is surrounded by a one
parameter family of periodic orbits, which are in turn bounded by an orbit
which is homoclinic to the saddle point $(0,0)$. These periodic orbits can be
parameterized by the ODE energy $E$ determined from the defining relation
\[
\frac{1}{2}(\partial_x u)^2 = E+\frac{1-c^2}{2}u^2-\frac{1}{p+2}u^{p+2}.
\]
The period $T(E,c)$ of these waves inside the homoclinic orbit satisfies
\[
\lim_{E\to 0^-}T(E,c)=+\infty,\quad\lim_{E\to E^*(c)^+} T(E,c)=\frac{2\pi}{\sqrt{(1-c^2)p}}
\]
where $E^*(c)$ is the ODE energy level associated with the equilibrium point
$((1-c^2)^{1/p},0)$. Below, we consider the stability of the periodic
traveling wave solutions of \eref{e:a51} in both the distinguished limits
$E\to 0^-$, corresponding to the solitary wave limit, and $E\to E^*(c)^+$,
associated with small amplitude periodic wave trains.

We begin by considering the stability of the periodic traveling waves of
\eref{e:a51} in the solitary wave limit.  When considering solitary wave
solutions which decay to zero as $x\to\pm\infty$ we obtain a two-parameter
family of solutions parameterized by wavespeed $c$ and spatial translation,
which corresponds to a one-parameter family of homoclinic orbits in the
two-dimensional phase space of \eref{e:53} (notice the boundary conditions in
this case require $a=E=b=0$).  Keeping throughout $a=b=0$ and fixing the
wavespeed $c\in(-1,1)$, by the above considerations there exists, up to
translations, a one parameter family of ``large period" periodic traveling
wave solutions of \eref{e:a51} parameterized by the ODE energy $E$.
Furthermore, for fixed $c$ these periodic waves approach locally uniformly on
$\mathbb{R}$ an appropriate translate of the limiting solitary wave as $E\to
E^-$.
%
The stability of the limiting solitary waves was investigated by
\citet{bona:geo88}, where they applied the abstract theory of
\citet{grillakis:sto87} to obtain nonlinear stability of the solitary wave
when $1<p<4$ and $p/4<c^2<1$.  This stability theory was later complimented
by \citet{liu}, obtaining nonlinear instability if either $c^2\leq p/4$ or
$p\geq 4$.  Next, we show that this phenomenon is observed again for the
``nearby" periodic traveling wave solutions of \eref{e:a51}.
For such periodic waves, it is easy to see that
\[
\rmn(\mathcal{L}_2)-\rmn(\langle\mathcal{L}_2^{-1}(1),1\rangle)=1,\quad0<-E\ll 1,
\]
and, furthermore,
\[
\lim_{\tilde E\to 0^-}D_{\mathrm{gKdV}}\Big{|}_{(a,E,c,b)=(0,\tilde E,c,0)}=-\Lambda(c,p)(4-p)
\]
for some positive constant $\Lambda(c,p)>0$ (for details, see the asymptotic
analysis in \citep[Section~3.2]{brj}).  It follows that for all\footnote{By
the previous example, we see that such a critical wavespeed also exists for
$p=4$.} $1\leq p<4$ there exists a critical wavespeed $c=c^*(p,E)$ such that
the periodic traveling wave $u(\cdot;a=0,E,c,b=0)$ with $0<-E\ll 1$ is
orbitally stable for $c^*(p,E)<|c|<1$ and spectrally unstable with $k_\rmr=1$
for $|c|<c^*(p,E)$, while for $p>4$ all such long-period waves waves are
spectrally unstable with $k_\rmr=1$.  This is consistent with the solitary
wave orbital stability/instability results of \citep{bona:geo88} and \citep{liu}.

Finally, continuing to restrict to $a=b=0$, we consider the stability of
small amplitude periodic traveling waves of \eref{e:a51} associated to those
periodic orbits of the profile ODE near the nonlinear center
$((1-c^2)^{1/p},0)$. Fixing the wavespeed $c\in(-1,1)$, it follows by basic
asymptotic analysis near the equilibrium solution $u=(1-c^2)^{1/p}$ that
\[
\rmn(\mathcal{L}_2)-\rmn(\langle\mathcal{L}_2^{-1}(1),1\rangle)=1
\]
and, furthermore, that at the nonlinear center we have
\[
D_{\mathrm{gKdV}}(a=0,E^*(c),c,b=0)=-\Omega(c,p)V''\left((1-c^2)^{1/p},0,1-c^2\right)^{-9/2}=-\Omega(c,p)\left((1-c^2)p\right)^{-9/2}
\]
for some positive constant $\Omega(c_0,p)>0$ (see \citep[Section~5]{Jkdv} for
details).
It follows that in a neighborhood of the equilibrium solution
$u\equiv(1-c^2)^{1/p}$ there exists a critical wavespeed $c(p,E)$ such that
all  nearby small-amplitude periodic traveling wave solutions of \eref{e:a51}
of the form $u(\cdot;a=0,E,c,b=0)$ with $E_0<E\ll 1$ are orbitally stable for
$c^*(p,E)<|c|<1$, and are spectrally unstable to perturbations with
$k_\rmr=1$ for $0<|c|<c^*(p,E)$.  In particular, the equilibrium solution
itself is orbitally stable for all $|c|\in(-1,1)$, implying that $c^*(p,E)\to
0^+$ as $E\to (E^*_0)^-$, i.e. as one approaches the equilibrium solution.
Note this is consistent with the numerical calculations of $c^*(p,E)$
presented in \citep{hakkaev:lsa12} in the cases $f(u)=u^2$ and $f(u)=u^3$.

\begin{remark}
It is interesting to note that when $p>4$ the waves with $a=b=0$ undergo a
transition to instability as one moves from a neighborhood of the equilibrium
solution $u\equiv(1-c^2)^{1/p}$ to a neighborhood of the limiting solitary
wave.
\end{remark}

\phantomsection                                         
\addcontentsline{toc}{section}{\refname}                


\begin{thebibliography}{34}
\providecommand{\natexlab}[1]{#1} \providecommand{\url}[1]{\texttt{#1}}
\expandafter\ifx\csname urlstyle\endcsname\relax
  \providecommand{\doi}[1]{doi: #1}\else
  \providecommand{\doi}{doi: \begingroup \urlstyle{rm}\Url}\fi

\bibitem[Angulo, Bona, and Scialom (2006)]{angulobona}
J. Angulo, J. Bona, and M. Scialom.
\newblock Stability of Cnoidal waves.
\newblock \emph{Advances Diff. Eq.}, 11:\penalty0 1321--1374, 2006).

\bibitem[Angulo and Quintero (2007)]{angulo}
J.~Angulo and J.R. Quintero.
\newblock Existence and orbital stability of cnoidal waves for a 1D Boussinesq equation.
\newblock \emph{Internat. J. Math. Math. Sci.}, 2007\penalty0 (2007):\penalty0 Article ID 52020, 36 pages
doi:10.1155/2007/52020.

\bibitem[Arruda(2009)]{arruda:nsp09}
L.~Arruda.
\newblock Nonlinear stability properes of periodic travelling wave solutions of the classical {K}orteweg-de {V}ries and {B}oussinesq equations.
\newblock \emph{Portugal. Math.}, 66\penalty0 (2):\penalty0 225--259, 2009.

\bibitem[Bona and Sachs(1988)]{bona:geo88}
J.~Bona and R.~Sachs.
\newblock Global existence of smooth solutions and stability of solitary waves
  for a generalized {B}oussinesq equation.
\newblock \emph{Comm. Math. Phys.}, 118:\penalty0 15--29, 1988.

\bibitem[Bronski, Johnson (2010)]{brj}
J.~Bronski and M.~Johnson.
\newblock The modulational instability for a generalized KdV equation.
\newblock \emph{Arch. Rat. Mech. Anal.} 197\penalty0 (2):\penalty0 357--400, 2010.

\bibitem[Bronski et~al.(2011)Bronski, Johnson, and Kapitula]{bronski:ait11}
J.~Bronski, M.~Johnson, and T.~Kapitula.
\newblock An index theorem for the stability of periodic traveling waves of
  {KdV} type.
\newblock \emph{Proc. Roy. Soc. Edinburgh: Section A}, 141\penalty0
  (6):\penalty0 1141--1173, 2011.

\bibitem[Chugunova and Pelinovsky(2009)]{chugunova:oqe09}
M.~Chugunova and D.~Pelinovsky.
\newblock On quadratic eigenvalue problems arising in stability of discrete
  vortices.
\newblock \emph{Lin. Alg. Appl.}, 431:\penalty0 962--973, 2009.

\bibitem[Deconinck and Kapitula()]{deconinck:ots10}
B.~Deconinck and T.~Kapitula.
\newblock On the spectral and orbital stability of spatially periodic
  stationary solutions of generalized {K}orteweg-de {V}ries equations.
\newblock submitted.

\bibitem[Deconinck and Kapitula(2010)]{deconinck:tos10}
B.~Deconinck and T.~Kapitula.
\newblock The orbital stability of the cnoidal waves of the {K}orteweg-de
  {V}ries equation.
\newblock \emph{Phys. Letters A}, 374:\penalty0 4018--4022, 2010.

\bibitem[Fang and Grillakis(1996)]{fang:eau96}
Y.~Fang and M.~Grillakis.
\newblock Existence and uniqueness for {B}oussinesq type equations on a circle.
\newblock \emph{Comm. Part. Diff. Eq.}, 21\penalty0 (7-8):\penalty0 1253--1277,
  1996.

\bibitem[Farah and Scialom(2010)]{farah:otp10}
L.~Farah and M.~Scialom.
\newblock On the periodic ``good" {B}oussinesq equation.
\newblock \emph{Proc. Amer. Math. Soc.}, 138\penalty0 (3):\penalty0 953--964,
  2010.

\bibitem[Grillakis et~al.(1987)Grillakis, Shatah, and Strauss]{grillakis:sto87}
M.~Grillakis, J.~Shatah, and W.~Strauss.
\newblock Stability theory of solitary waves in the presence of symmetry, {I}.
\newblock \emph{Journal of Functional Analysis}, 74:\penalty0 160--197, 1987.

\bibitem[Grillakis et~al.(1990)Grillakis, Shatah, and Strauss]{grillakis:sto90}
M.~Grillakis, J.~Shatah, and W.~Strauss.
\newblock Stability theory of solitary waves in the presence of symmetry, {II}.
\newblock \emph{Journal of Functional Analysis}, 94:\penalty0 308--348, 1990.

\bibitem[Gurski et~al.(2004)Gurski, Koll\'ar, and Pego]{gurski:sdo04}
K.~Gurski, R.~Koll\'ar, and R.~Pego.
\newblock Slow damping of internal waves in a stably stratified fluid.
\newblock \emph{Proc. Roy. Soc. Lond. A}, 460:\penalty0 977--994, 2004.

\bibitem[Hakkaev et~al.(2012)Hakkaev, Stanislavova, and
  Stefanov]{hakkaev:lsa12}
S.~Hakkaev, M.~Stanislavova, and A.~Stefanov.
\newblock Linear stability analysis for periodic traveling waves of the
  {B}oussinesq equation and the {KGZ} system.
\newblock preprint, 2012.

\bibitem[H\v{a}r\v{a}gu\c{s} and Gallay(2008)]{HG07}
M.~H\v{a}r\v{a}gu\c{s} and T.~Gallay.
\newblock Orbital stability of periodic waves for the nonlinear Schrödinger equation.
\newblock \emph{J. Dyn. Diff. Eqns.}, 19:\penalty0 825--865, 2007.

\bibitem[H\v{a}r\v{a}gu\c{s} and Kapitula(2008)]{haragus:ots08}
M.~H\v{a}r\v{a}gu\c{s} and T.~Kapitula.
\newblock On the spectra of periodic waves for infinite-dimensional
  {H}amiltonian systems.
\newblock \emph{Physica D}, 237\penalty0 (20):\penalty0 2649--2671, 2008.

\bibitem[Johnson(2009)]{Jkdv}
M.~Johnson.
\newblock Nonlinear stability of periodic traveling wave solutions of the generalized Korteweg-de Vries equation.
\newblock \emph{SIAM J. Math. Anal.}, 41\penalty0 (5):\penalty0 1921--1947, 2009.

\bibitem[Johnson(2010)]{Jbbm}
M.~Johnson.
\newblock On the stability of periodic solutions of the generalized Benjamin-Bona-Mahony equation .
\newblock \emph{Physica D}, 239\penalty0 (19):\penalty0 1892-1908, 2010.

\bibitem[Kapitula(2007)]{kapitula:ots07}
T.~Kapitula.
\newblock On the stability of ${N}$-solitons in integrable systems.
\newblock \emph{Nonlinearity}, 20\penalty0 (4):\penalty0 879--907, 2007.

\bibitem[Kapitula(2010)]{kapitula:tks10}
T.~Kapitula.
\newblock The {K}rein signature, {K}rein eigenvalues, and the {Krein
  Oscillation Theorem}.
\newblock \emph{Indiana U. Math. J.}, 59:\penalty0 1245--1276, 2010.

\bibitem[Kapitula and Promislow(2012{\natexlab{a}})]{kapitula:ait12}
T.~Kapitula and K.~Promislow.
\newblock \emph{An Introduction to Spectral Stability: Theory and
  Applications}.
\newblock Springer-Verlag, 2012{\natexlab{a}}.
\newblock in preparation.

\bibitem[Kapitula and Promislow(2012{\natexlab{b}})]{kapitula:sif12}
T.~Kapitula and K.~Promislow.
\newblock Stability indices for constrained self-adjoint operators.
\newblock \emph{Proc. Amer. Math. Soc.}, 140\penalty0 (3):\penalty0 865--880,
  2012{\natexlab{b}}.

\bibitem[Kapitula et~al.(2004)Kapitula, Kevrekidis, and
  Sandstede]{kapitula:cev04}
T.~Kapitula, P.~Kevrekidis, and B.~Sandstede.
\newblock Counting eigenvalues via the {K}rein signature in
  infinite-dimensional {H}amiltonian systems.
\newblock \emph{Physica D}, 195\penalty0 (3\&4):\penalty0 263--282, 2004.

\bibitem[Kapitula et~al.(2005)Kapitula, Kevrekidis, and
  Sandstede]{kapitula:ace05}
T.~Kapitula, P.~Kevrekidis, and B.~Sandstede.
\newblock Addendum: {C}ounting eigenvalues via the {K}rein signature in
  infinite-dimensional {H}amiltonian systems.
\newblock \emph{Physica D}, 201\penalty0 (1\&2):\penalty0 199--201, 2005.

\bibitem[Koll\'ar(2011)]{kollar:hmf11}
R.~Koll\'ar.
\newblock Homotopy method for nonlinear eigenvalue pencils with applications.
\newblock \emph{SIAM J. Math. Anal.}, 43\penalty0 (2):\penalty0 612--633, 2011.

\bibitem[Kostenko(2002)]{kostenko:otd02}
A.~Kostenko.
\newblock On the defect index of quadratic self-adjoint operator pencils.
\newblock \emph{Math. Notes}, 72\penalty0 (2):\penalty0 285--290, 2002.


\bibitem[Liu(1993)]{liu}
Y.~Liu.
\newblock Instability of solitary waves for generalized Boussinesq equations.
\newblock \emph{J. Dynam. Differential Equations}, 5:\penalty0 537--558, 1993.

\bibitem[Lyong(1993)]{lyong:tsp93}
N.~Lyong.
\newblock The spectral properties of a quadratic pencil of operators.
\newblock \emph{Russ. Math. Surv.}, 48:\penalty0 180--182, 1993.

\bibitem[Markus(1988)]{markus:itt88}
A.~Markus.
\newblock Introduction to the spectral theory of polynomial operator pencils.
\newblock In \emph{Translation of Mathematical Monographs}, volume~71. Amer.
  Math. Soc., 1988.

\bibitem[Pelinovsky(2005)]{pelinovsky:ilf05}
D.~Pelinovsky.
\newblock Inertia law for spectral stability of solitary waves in coupled
  nonlinear {S}chr\"odinger equations.
\newblock \emph{Proc. Royal Soc. London A}, 461:\penalty0 783--812, 2005.

\bibitem[Pivovarchik(2007)]{pivovarchik:oso07}
V.~Pivovarchik.
\newblock On spectra of a certain class of quadratic operator pencils with
  one-dimensional linear part.
\newblock \emph{Ukrainian Math. J.}, 59\penalty0 (5):\penalty0 766--781, 2007.

\bibitem[Shkalikov(1996)]{shkalikov:opa96}
A.~Shkalikov.
\newblock Operator pencils arising in elasticity and hydrodynamics: the
  instability index formula.
\newblock In I.~Gohberg, P.~Lancaster, and P.~Shivakumar, editors, \emph{Recent
  Developments in Operator Theory and its Applications}, volume~87 of
  \emph{Operator Theory Advances and Applications}, pages 358--385, 1996.

\bibitem[Stanislavova and Stefanov (2012)]{stanis}
M.~Stanislavova and A.~Stefnov.
\newblock Linear stability analysis for traveling waves of second order in time PDE's.
\newblock preprint, 2012.
\end{thebibliography}
\bibliographystyle{plainnat}


\end{document}